\documentclass[12pt, reqno, twoside, letterpaper]{amsart}

\usepackage{amsmath,amssymb,amsbsy,amsfonts,amsthm,latexsym,
            amsopn,amstext,amsxtra,euscript,amscd,stmaryrd,mathrsfs,
            cite,array}

\usepackage{graphicx}

\numberwithin{equation}{section}

\usepackage{color}
\usepackage[usenames,dvipsnames,svgnames,table]{xcolor}

\def\eps{\varepsilon}
\def\mand{\qquad\mbox{and}\qquad}
\def\fl#1{\left\lfloor#1\right\rfloor}

\def\({\left(}
\def\){\right)}

\newcommand{\e}{\ensuremath{\mathbf{e}}}

\newcommand{\cB}{\ensuremath{\mathcal{B}}}

\newcommand{\cF}{\ensuremath{\mathcal{F}}}

\newcommand{\cL}{\ensuremath{\mathcal{L}}}
\newcommand{\cM}{\ensuremath{\mathcal{M}}}
\newcommand{\cN}{\ensuremath{\mathcal{N}}}

\newcommand{\cP}{\ensuremath{\mathcal{P}}}

\newcommand{\cS}{\ensuremath{\mathcal{S}}}

\newcommand{\cT}{\ensuremath{\mathcal{T}}}

\newcommand{\RR}{\ensuremath{\mathbb{R}}}
\newcommand{\ZZ}{\ensuremath{\mathbb{Z}}}

\usepackage{color}
\usepackage{hyperref}
\hypersetup{
    pdftoolbar=true,
    pdfmenubar=false,
    pdffitwindow=false,
    pdfstartview={FitH},
    pdftitle={1}, 
    pdfauthor={Victor Z. Guo}, 
    pdfcreator={Victor Z. Guo}, 
    pdfproducer={Victor Z. Guo}, 
    pdfnewwindow=true,
    colorlinks=true,
    linkcolor={black},
    citecolor={black},
    filecolor={black},
    urlcolor={black},
}

\usepackage[margin=2.5cm]{geometry}

\usepackage{amsthm}
\newtheoremstyle{customthm}
{1em}                    
{1em}                    
{\itshape}               
{}                       
{\scshape}               
{.}                      
{5pt plus 1pt minus 1pt} 
{}                       

\newtheoremstyle{customrem}
{1em}                    
{1em}                    
{}                       
{}                       
{\scshape}               
{.}                      
{5pt plus 1pt minus 1pt} 
{}                       

\theoremstyle{customthm}

\newtheorem{X}{X}[section]
\newtheorem{theorem}[X]{Theorem}

\newtheorem{lemma}[X]{Lemma}

\newtheorem{proposition}[X]{Proposition}

\theoremstyle{customrem}

\usepackage{etoolbox}
\AtEndEnvironment{remark}{\null\hfill\qedsymbol}

\AtEndEnvironment{definition}{\null\hfill\qedsymbol}

\renewcommand{\le}{\ensuremath{\leqslant}}
\renewcommand{\ge}{\ensuremath{\geqslant}}

\renewcommand{\pod}[1]{\mathchoice
  {\allowbreak \if@display \mkern 5mu\else \mkern 5mu\fi (#1)}
  {\allowbreak \if@display \mkern 5mu\else \mkern 5mu\fi (#1)}
  {\mkern4mu(#1)}
  {\mkern4mu(#1)}
}

\DeclareSymbolFont{EUEX}{U}{euex}{m}{n}

\DeclareSymbolFont{euexlargesymbols}{U}{euex}{m}{n}
\DeclareMathSymbol{\intop}{\mathop}{euexlargesymbols}{"52}
     \def\int{\intop\nolimits}

\DeclareSymbolFont{euexsymbols}     {U}{euex}{m}{n}
\DeclareMathSymbol{\smallint}{\mathop}{euexsymbols}{"52}

\allowdisplaybreaks

\title[Exponential sums with polynomials]{Exponential sums with polynomials and their applications to primes in sparse sets}

\author[Lingyu Guo]{Lingyu Guo}

\address{School of Mathematics and Statistics, Xi'an Jiaotong University, Xi'an, Shaanxi, China.}

\email{guo.lingyu@foxmail.com}

\author[Victor Z. Guo]{Victor Zhenyu Guo}

\address{School of Mathematics and Statistics, Xi'an Jiaotong University, Xi'an, Shaanxi, China.}

\email{vzguo@foxmail.com; guozyv@xjtu.edu.cn}

\author[Mengyao Jing]{Mengyao Jing}

\address{Research Center for Number Theory and Its Applications, School of Mathematics, Northwest University, Xi'an Shaanxi, China.}

\email{myjing@nwu.edu.cn}

\date{\today}

\begin{document}

\begin{abstract}
Exponential sums with monomials are highly related to many interesting problems in number theory and well studied by many literatures. In this paper, we consider the exponential sums with polynomials and prove a new upper bound. As an application, we study the Piatetski-Shapiro sequence of the form $(\lfloor n^c \rfloor)$ where $c > 1$ is not an integer. We improve the admissible range of the asymptotic formula for primes in the intersection of Piatetski-Shapiro sequences. We also study the iterated Piatetski-Shapiro sequence and prove an asymptotic formula for the prime counting function.
\end{abstract}

\maketitle

\begin{quote}
\end{quote}

\begin{quote}
\end{quote}



\tableofcontents


\section{Introduction}
\label{sec:introduction}

The estimation of the exponential sum of the following form is a core part of many interesting problems in number theory:
$$
\cS_0 = \sum_{m \sim M}  \sum_{n \sim N} a_{m} b_n \e(f(m,n))
$$
with the notations $\e(t) = e^{2\pi it}$, where $N, M$ are positive numbers, $a_{m}, b_n$ are complex numbers, $f$ is a smooth function, $m \sim M$ means that $M < m \le 2M$.

The monomial case when
$$
f(m,n) =    E\frac{m^\beta n^\gamma}{M^\beta N^\gamma}\ \ \ \ (E\in\mathbb{R})
$$
was well studied; see Bombieri and Iwaniec \cite{BI1986}, Fouvry and Iwaniec \cite{FI1989},  and so on. We mention that a remarkable result is by Robert and Sargos \cite{RS2006} who proved the bound of $\cS_0$ via double large sieves and a spacing problem solved by an optimal bound of the number of solutions of a diophantine inequality.

In this paper, we consider a more general case when the function $f$ is a polynomial, which leads the exponential sum to be
$$
\cS= \sum_{m\sim M}\mathop{\sum_{n\sim N}}_{mn\sim X}a_mb_n
\mathbf{e}\left(\sum_{1\leqslant j\leqslant d}
\frac{E_jm^{\beta_j}n^{\gamma_j}}{M^{\beta_j}N^{\gamma_j}}\right).
$$
Unfortunately the method for the monomial case in \cite{BI1986, FI1989, RS2006} does not work in the polynomial case, since double large sieve provides a bound worse than the trivial bound. A traditional method is to apply the van der Corput's inequality or a higher derivative test for exponential sums which work on any smooth function; see Chapter 2 in \cite{GrahamK1991}. With this idea, Baker \cite{Baker2014} proved a bound of a special case of $\cS$ when $\beta_j = \gamma_j$ for all $j$.

We prove the bound of $\cS$ which is also better than Baker's bound in the special case. A key idea is that we recognize the method of exponent pairs (Lemma  \ref{lem: exponent pair method}) can be applied to this problem.

\begin{theorem}\label{theorem: 1}
Let $M,N,X$ be positive real numbers such that $X=MN$.
Let $E_j,\beta_j,\gamma_j\ (j=1,\cdots,d)$ be real numbers.
For every exponent pair $(\kappa,\lambda)$ and any sufficiently small real
$\varepsilon>0$ we have
\begin{align*}
X^{-3\varepsilon}\cS\ll M^{1/2}X^{1/2}+\Delta^{-1/2}X
&+M^{(\lambda-\kappa-1)/(2\kappa+2)}X\\
&+\Delta^{\kappa/(2\kappa+2)}M^{(\lambda-1)/(2\kappa+2)}X^{(2+\kappa)/(2\kappa+2)},
\end{align*}
where $|a_m|\ll X^{\varepsilon}$, $|b_n|\ll X^{\varepsilon}$ and
$$
\Delta=\Bigg|\sum_{1\leqslant j\leqslant d}\gamma_j\cdots(\gamma_j-s+1)E_j\Bigg|.
$$
\end{theorem}

\subsection{Applications of the exponential sums}
An application is related to the Piatetski-Shapiro sequences defined as
$$
\mathcal{N}^c=(\lfloor n^c\rfloor)_{n=1}^{\infty},
$$
where $\lfloor x\rfloor$ is the largest integer not exceeding $x$ with real $c>1$ and $c\not\in\mathbb{N}$. Piatetski-Shapiro \cite{Shapiro1953} gave the following prime number theorem:
\begin{align}\label{eq: Piatetski-Shapiro}
\#\left\{p\leqslant X: p\in \mathcal{N}^c\right\}
=(1+o(1))\frac{X^{1/c}}{c\log X}\ \ \ \hbox{as}\ X\rightarrow\infty
\end{align}
for $1<c<12/11$. This result is considered as an approximation of the well-known conjecture that there exist infinitely many primes of the form $n^2+1$.

To date, the best known admission range for $c$ is $1<c<2817/2426$  due to  Rivat and Sargos \cite{RivatS2001}. If one only considers (\ref{eq: Piatetski-Shapiro}) with a lower bound of the right order of magnitude instead of an asymptotic formula, a large range $1<c< 243/205$ was proved by Rivat and Wu \cite{RivatW2001}. The readers are referred to \cite{Guo2021} for a survey regarding the distribution of Piatetski-Shapiro primes.

An analogue problem of (\ref{eq: Piatetski-Shapiro}) is to count primes in the intersection of multiple Piatetski-Shapiro sequences, which is
$$
\mathcal{N}(X,\boldsymbol{c})=\#\left\{p\leqslant X: p=\lfloor n_1^{c_1}\rfloor
=\cdots=\lfloor n_d^{c_d}\rfloor \hbox{\ for some integers}\ n_1,\cdots,n_d \right\},
$$
where $c_1,\cdots,c_d>1$ and $\boldsymbol{c}=(c_1,\cdots,c_d)$. Denote
$$
\gamma_j=\frac{1}{c_j},\ \ \delta_j=1-\gamma_j,\ \ \sigma=\delta_1+\cdots+\delta_d.
$$
We expect that for some $\sigma_d$ such that $\sigma<\sigma_d$, it follows that
\begin{align}\label{eq: Baker}
	\mathcal{N}(X,\boldsymbol{c})=(1+o(1))\frac{\gamma_1\cdots\gamma_d}
	{1-\sigma}\frac{X^{1-\sigma}}{\log X}.
\end{align}
Leitmann \cite{Leitmann1982} proved that $\sigma_2 = 1/28$ as a first result.  Sirota \cite{Sirota1983} extends this result to $\sigma_2 = 1/16$. Moreover, Sirota also proved an asymptotic formula for these primes in arithmetic progressions when $d=2$.  Zhai \cite{Zhai1999} proved that $\sigma_d = d/(4d^2+2)$ for $d \ge 3$ and claimed a stronger result for $d=2$ but it has not been published. Baker \cite{Baker2014} proved the following best result with
$$
\sigma_2=\sigma_3=\frac{1}{12},\ \ \sigma_4=\frac{1}{16},\ \ \sigma_5=\frac{1}{18},\ \
\sigma_d=\frac{1}{3d}\ (d\geqslant6).
$$

In this paper, we further improve Baker's results as an application of our Theorem \ref{theorem: 1}. Our main result is as follows.

\begin{theorem}\label{theorem: 2}
Under the above notations, suppose that $\sigma<\sigma_d$, where
$$
\sigma_d=\frac{290}{3297}\approx\frac{1}{11.36}\ (2\leqslant d\leqslant 10)\ \ \ \hbox{and}\ \ \
\sigma_d=\frac{1}{d+1}\ (d\geqslant 11),
$$
we have
$$
\mathcal{N}(X,\boldsymbol{c})=(1+o(1))\frac{\gamma_1\cdots\gamma_d}
{1-\sigma}\frac{X^{1-\sigma}}{\log X}.
$$
\end{theorem}

Another application of Theorem \ref{theorem: 1} is the iterated Piatetski-Shapiro sequences of the form
$$
\mathcal{N}^{(c_1,c_2)} = \big( \lfloor\lfloor n^{c_1}\rfloor ^{c_2}\rfloor\big)_{n=1}^{\infty},
$$
where $1<c_1,c_2<2$ are real numbers. Obviously, the elements in $\mathcal{N}^{(c_1,c_2)}$
construct new sparser sets. It is natural to establish a prime number theorem for the iterated Piatetski-Shapiro sequences. Our result is enunciated as follows.

\begin{theorem}\label{theorem: 3} Let $1<c_1,c_2<2$ be real numbers and denote
$\gamma_1=1/c_1$ and $\gamma_2=1/c_2$. Denote the cardinality
$$
\nonumber \cP(X; c_1,c_2)=
\#\big\{p\leqslant X: p=\lfloor\lfloor h^{c_1}\rfloor ^{c_2}\rfloor \ \hbox{for some integer}
\ h\geqslant 1\big\}.
$$
If $\gamma_1$ and $\gamma_2$ satisfy that
\begin{align*}
26\gamma_1\gamma_2-2\gamma_2-23>0,
\end{align*}
we have
$$
\nonumber \cP(X; c_1,c_2)=(1+o(1))\frac{X^{\gamma_1\gamma_2}}{c_1c_2\log X}.
$$
\end{theorem}

The following figure (Figure \ref{fig:c1c2}) provides the admissible range of $c_1$ and $c_2$.

\begin{figure}[htbp]
  \centering
  \includegraphics[width=0.75\textwidth]{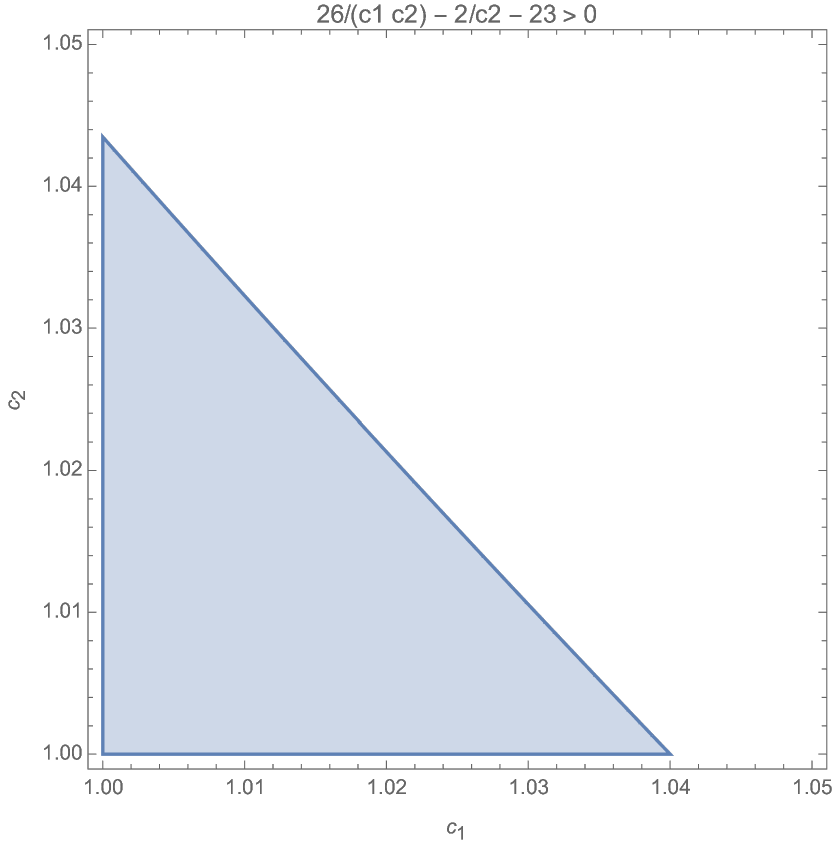}
  \caption{The admissible range of $c_1$ and $c_2$}\label{fig:c1c2}
\end{figure}

\subsection{Notations}

As is customary, we write
$$
\e(x)= \exp(2\pi ix)\mand
\{ x \} = x-\fl{x}\qquad(x\in\RR).
$$
We denote by $\Lambda(n)$ the Mangoldt function and denote by
$\|x\|$ the distance to the nearest integer of $x$.
The sawtooth function is defined by
$$
\psi(x) = x-\fl{x}-\tfrac12=\{x\}-\tfrac12\qquad(x\in\RR).
$$


For given functions $F$ and $G$, the notations $F\ll G$ and $F=O(G)$ are equivalent to the statement that the inequality $|F|\le C|G|$ holds with some constant $C>0$.
We write $F\asymp G$ to indicate $F\ll G\ll F$.
We also let $\eps>0$ be a sufficiently small real number throughout this paper.

\section{Outlines}

Our improvements on exponential sums relies on an observation that we can apply exponent pair method together with a counting problem (Lemma \ref{lemma: zhai's lemma}). We complete the proof of Theorem \ref{theorem: 1} in Section \ref{section: bounds of es}.

The applications to prime numbers are started by a lattice method investigated by Harman \cite{Harman1993} and Baker \cite{Baker2014}. We provide a slightly more general form here.  The theorem is proved in Section \ref{subsection: polar Lattice}.

\begin{theorem}\label{theorem: polar}
	Let $\Gamma$ denote a lattice in $\RR^d$ and $\Pi$ denote the polar lattice of $\Gamma$. Let $a_1,\cdots,a_N$ be complex weights, and let $\alpha_1,\cdots,\alpha_N$ be a sequence in
	$\mathbb{R}^d$. Define
	$$
	\cN = \sum_{\boldsymbol{\alpha}_n \in \cB_d \pmod{\Gamma}} a_n,
	$$
	where $\boldsymbol{\alpha}_n=(\alpha_1,\cdots,\alpha_N)$ and $\cB_d=[-1,1]^d$. Then we have
	$$
	\cN = \frac{1}{\det\Gamma} \bigg( (2^d + O(w^{-1}))
	\sum_{n\leqslant N}a_n + O\bigg( \sum_{\substack{\boldsymbol{q}\in\Pi \\ 0<|\boldsymbol{q}|<w}} \bigg| \sum_{n\leqslant N}a_ne(\boldsymbol{q \alpha}_n) \bigg| \bigg) \bigg),
	$$
	where $w$ is a parameter in $(1,\infty)$ at our disposal.
\end{theorem}

\subsection{Intersection of Piatetski-Shapiro sequences}
Based on the lattice method, we prove Proposition \ref{theorem: characterization} to analyze primes in the intersection of Piatetski-Shapiro sequences, which means that it is sufficient to bound
$$
\mathop{\sum_{\boldsymbol{h}\in\mathbb{Z}^d\setminus\{\boldsymbol{0}\}}}
_{|h_j|<\mathcal{L}X^{\delta_j}}\Bigg|\sum_{X < n \leqslant aX}\Lambda(n)
\mathbf{e}\left(h_1n^{\gamma_1}+\cdots+h_dn^{\gamma_d}\right)\Bigg|
$$
with $\boldsymbol{h}=(h_1,\cdots,h_d)$ and $\cL = \log X$.
Now via the well-known Heath-Brown identity (Lemma~\ref{lem: Heath-Brown inentity}), we need to work on the following sum
$$
\mathfrak{S}=\mathop{\sum_{0<|h_j|<X^{\delta_j}\mathcal{L}}}_{j=1,\cdots,d}
\Bigg|\mathop{{\sum_{m \sim M} \sum_{n \sim N}}}_{mn \sim X} a_m b_n \e \left(
h_1(mn)^{\gamma_1}+\cdots+h_d(mn)^{\gamma_d}\right)\Bigg|,
$$
where $N \ge 1, M \ge 1, MN \asymp X$, $|a_m|\ll X^{\varepsilon}$ and $|b_n|\ll X^{\varepsilon}$
for any $\varepsilon>0$. If $b_n=1$ or $b_n=\log n$, we call it a Type I sum and denote it as $\mathfrak{S}_{I}$; otherwise we call it as a Type II sum and denote it as $\mathfrak{S}_{II}$.

A normal method is to bound type I sum by the van der Corput inequality or a higher derivative test. The type II sum is estimated by the Weyl-van der Corput inequality (A-process) together with the van der Corput inequality. We mention that Zhai \cite{Zhai1999} and Baker \cite{Baker2014} investigated the problem exactly by the normal method we mention here.

In this paper, we bound the type II sum by our Theorem \ref{theorem: 1}. The type I sum is bounded by an exponent pair method. Finding that how to apply an exponent pair method is one of the main reasons leading to the improvements. Section \ref{section: Sums of an arithmetic function} gives a detailed proof how we apply the lattice method to the intersection of Piatetski-Shapiro sequences and Section \ref{section: intersection P-S primes} provides the estimation of exponential sums based on Theorem \ref{theorem: 1}.

\subsection{The iterated Piatetski-Shapiro sequence}

We also deploy primes in the iterated Piatetski-Shapiro sequence by our lattice method (Theorem \ref{theorem: polar}) and arrive at a sum of the form
$$
\sum_{\substack{(k,h) \in \ZZ^d \backslash (0,0) \\ |k|<\cL X^{1-\gamma_2} \\ |h|<\cL X^{\gamma_2(1-\gamma_1)}}} \Bigg| \sum_{X < m \leqslant aX}\Lambda(m) \e (km^{\gamma_2} + h\lfloor(m+1)^{\gamma_2}\rfloor^{\gamma_1}) \Bigg|.
$$
The trick is that we need a Fourier analysis method to work on exponential sums with floor functions. We introduce a method due to Kolesnik \cite[Lemma 1]{Kolesnik2003}.

\begin{lemma}\label{lemma: Kolesnik}
	Let $f(\mu,\nu)$ be a real valued function such that
	$$
	|f(\mu_1,\nu)-f(\mu_2,\nu)|\leqslant \lambda |\mu_1-\mu_2|.
	$$
	Then, for any real function $g(x)$, any positive integer $r$ and any $W>0$
	we have
	\begin{align*}
		\sum_{x}a(x)&\mathbf{e}\left(f(g(x),\{h(x)\})\right)\\
		&=\frac{1}{W}\sum_{w=0}^{W-1}\sum_{j=0}^{\infty}B_{j,w}\sum_{x}a(x)
		\mathbf{e}\left(f\left(g(x),\frac{w}{W}\right)+jh(x)\right)\\
		&\quad+O\left(\frac{\lambda r+r}{W}\sum_{x}|a(x)|\right)\\
		&\quad+O\left(\frac{r}{W}\sum_{j=0}^{\infty}\frac{\sin(2\pi rj/W)}{\sin(\pi j/W)}
		A_j\sum_{x}|a(x)|\mathbf{e}(jh(x))\right),
	\end{align*}
	where
	\begin{align*}
		&A_j=\left(\frac{\sin(\pi j/M)}{\pi j/W}\right)^{r+1},\ \ \ A_0=1\\
		&B_{j,w}=A_j\mathbf{e}\left(-\frac{(2w+1)j}{2W}\right).
	\end{align*}
\end{lemma}

After many technical calculations, it is sufficient to prove
$$
\Psi =\sum_{X < m \leqslant aX}
\Lambda(m)\mathbf{e}\left(h\left((m+1)^{\gamma_2}-\frac{w}{W}\right)^{\gamma_1}
+km^{\gamma_2}+j(m+1)^{\gamma_2}\right).
$$
Now the Vaughan identity (Lemma~\ref{lem:vaughan}) also allows us to bound corresponding type I and type II sums, which are similar to the proof of Theorem \ref{theorem: 1}. We make the technical discussion together with the proof of Theorem \ref{theorem: 3} in Section \ref{section: proof 3} and finish the bounds of error terms in Section \ref{section: 9}.

In this article, we only consider the twice-iterated Piatetski-Shapiro sequences. We believe that our method also works for multiple iterated Piatetski-Shapiro sequences of the form
$$
\mathcal{N}^{(\mathbf{c})} = \big(\fl{\fl{\cdots{\fl{n^{c_1}}^{c_2}}\cdots}^{c_n}}\big)_{n=1}^\infty.
$$
However, this may be done by a very complicated induction which can take very long calculations. We think this may be of independent interest.

\section{Preliminary}\label{section: Preliminary}

The following lemma is due to Zhai \cite[Proposition 1]{Zhai1999}.
\begin{lemma}\label{lemma: zhai's lemma}
Let $d\geqslant 2$ be a fixed integer. Suppose that $a_1,\cdots,a_d$ are non-zero real numbers,
$\alpha_1,\cdots,\alpha_d$ are distinct real constants such that $\alpha_1,\cdots,\alpha_d\not\in\mathbb{Z}$,
and assume that $\Theta>0$. Let $\mathcal{I}\subset[1,2]$ be a subinterval such that
$$
\left|a_1 t^{\alpha_1} + \cdots+a_d t^{\alpha_d} \right|\leqslant \Theta,\ \ \ t\in\mathcal{I}.
$$
Then we have
$$
|\mathcal{I}|\ll \left(\frac{\Theta}{|a_1|+\cdots+|a_d|}\right)^{1/(d-1)}.
$$
\end{lemma}

We provide the method of exponent pair by the following lemma.

\begin{lemma}\label{lem: exponent pair method}
Let $s\geqslant 2$ be a positive integer, and let $f(x)$ be a real valued
function with $s$ continuous derivatives on $[N,2N]$ such that
$$
|f^{(s)}(n)|\asymp YN^{1-s}.
$$
Then we have
$$
\sum_{n\sim N}\mathbf{e}(f(n))\ll Y^{\kappa}N^{\lambda}+Y^{-1},
$$
where $(\kappa,\lambda)$ is any exponent pair.
\end{lemma}

\begin{proof}
See \cite[Chapter 3]{GrahamK1991} or \cite[Lemma 1]{CaoZ1998}.
\end{proof}

Next, we recall the following identity for the von Mangoldt function $\Lambda$, which is due to Vaughan (see Davenport~\cite[p.~139]{Daven}).

\begin{lemma}
	\label{lem:vaughan}
	Let $U,V\ge 1$ be real parameters.  For any $n>U$ we have
	$$
	\Lambda(n)=
	-\sum_{k\,|\,n}a(k)
	+\sum_{\substack{cd=n\\d\le V}}(\log c)\mu(d)
	-\sum_{\substack{kc=n\\k>1\\c>U}}\Lambda(c)b(k),
	$$
	where
	$$
	a(k)=\sum_{\substack{cd=k\\c\le U\\d\le V}}\Lambda(c)\mu(d)
	\mand b(k)=\sum_{\substack{d\,\mid\,k\\d\le V}}\mu(d)
	$$
\end{lemma}

A direct application of the Vaughan identity is to decompose the sum $\sum_{n \le x} \Lambda(n) f(n)$ into sums of type I and type II sums of the form
$$
\sum_{\substack{m \le M \\ mn \le x}} a_m f(mn) \qquad \text{and} \qquad \sum_{\substack{K \le m < 2K \\ mn \le x}} a_m b_n f(mn)
$$
respectively where
$$
M \le \max\{x^{1-\beta}, x^{2\beta}\} \qquad \text{and} \qquad x^\beta \le K \le x^{1-\beta}
$$
for some $\beta \in (0,1/2)$. The coefficients $a_m, b_n$ are bounded by $x^\eps$. Usually we choose $\beta = 1/3$.

We also have an alternative decomposition of the von Mangoldt function due to Heath-Brown.

\begin{lemma}\label{lem: Heath-Brown inentity}
Let $z\geqslant1$ be a real number and let $k\geqslant1$ be an integer.
For any integer $n\leqslant2z^k$, there holds
$$
\Lambda(n)=\sum_{j=1}^{k}(-1)^{j-1}\binom{k}{j}
\mathop{\sum_{n_1\cdots n_{2j}=n}}_{n_{j+1},\cdots,n_{2j}\leqslant z}
(\log n_1)\mu(n_{j+1})\cdots\mu(n_{2j}).
$$
\end{lemma}
\begin{proof}
See the arguments on pp. 1366--1367 of Heath-Brown \cite{Heath-Brown}.
\end{proof}

\begin{lemma} \label{lemmea: balance} Let
$$
L(E)=\sum_{1\leqslant i\leqslant u}A_iE^{a_i}+\sum_{1\leqslant j\leqslant v}B_jE^{-b_j},
$$
where $A_i,B_j,a_i$ and $b_j$ are positive. Let $0\leqslant E_1\leqslant E_2$. Then there is
some $E\in(E_1,E_2]$ such that
$$
L(E)\ll \sum_{1\leqslant i\leqslant u}\sum_{1\leqslant j\leqslant v}\left(A_i^{b_j}B_j^{a_i}\right)^{1/(a_i+b_j)}
+\sum_{1\leqslant i\leqslant u}A_iE_1^{a_i}+\sum_{1\leqslant j\leqslant v}B_jE_2^{-b_j},
$$
where the implied constant only depends on $u$ and $v$.
\end{lemma}

\begin{proof}
See Lemma 2.4 in \cite{GrahamK1991}.
\end{proof}

The following lemma is the famous Weyl-van der Corput inequality, also called the A-process; see \cite[Lemma 2.5]{GrahamK1991}.

\begin{lemma}
\label{lem:A}
Suppose $f(n)$ is a complex valued function and $I$ is an interval such that $f(n) = 0$ if $n \notin I$. If $Q$ is a positive integer then
$$
|\sum_{n \in I} f(n)|^2 \le \frac{|I| + Q}{Q} \sum_{|q|< Q} \big(1 - \frac{q}{Q} \big) \sum_{n \in I} f(n) \overline{f(n-q)}.
$$
\end{lemma}

\section{Bounds of exponential sums}\label{section: bounds of es}

In order to prove Theorem \ref{theorem: 1}, we first give the following lemma.

\begin{lemma}\label{lemma: a lemmma to estimate S}
Let $d\geqslant 1$ be an integer, and let $L,L_0,z_1,\cdots,z_d,\alpha_1,\cdots,
\alpha_d$ be real numbers such that $L>1$, $L_0\leqslant L$ and
$\alpha_1,\cdots,\alpha_d\in(0,1)$. Then we have
$$
\sum_{L<l\leqslant L+L_0}\mathbf{e}\left(z_1l^{\alpha_1}+\cdots+z_dl^{\alpha_d}\right)
\ll \Delta_0^{\kappa}L^{\lambda}+\Delta_0^{-1},
$$
where
$$
\Delta_0=\Bigg|\sum_{1\leqslant j\leqslant d}\alpha_j\cdots(\alpha_j-s+1)z_jL^{\alpha_j-1}\Bigg|.
$$
\end{lemma}
\begin{proof}
Write $f(x)=z_1x^{\alpha_1}+\cdots+z_dx^{\alpha_d}$. For any positive integer $s$,
we have
\begin{align*}
|f^{(s)}(l)|=\Bigg|\sum_{1\leqslant j\leqslant d}\alpha_j\cdots(\alpha_j-s+1)
z_jl^{\alpha_j-s}\Bigg|\asymp \Delta_0 L^{1-s}.
\end{align*}
It follows from Lemma \ref{lem: exponent pair method} that
$$
\sum_{L<l\leqslant L+L_0}\mathbf{e}\left(f(l)\right)
\ll \Delta_0^{\kappa}L^{\lambda}+\Delta_0^{-1}.
$$
This proves Lemma \ref{lemma: a lemmma to estimate S}.
\end{proof}

Now we return to the proof of Theorem \ref{theorem: 1}.
Applying the Cauchy inequality we have
\begin{align*}
X^{-2\varepsilon}|\cS|^2\ll M\sum_{m\sim M}\Bigg|\mathop{\sum_{n\sim N}}_{mn\sim X}b_n
\mathbf{e}\left(\sum_{1\leqslant j\leqslant d}
\frac{E_jm^{\beta_j}n^{\gamma_j}}{M^{\beta_j}N^{\gamma_j}}\right)\Bigg|^2.
\end{align*}
Let $Q\leqslant N$ be a positive real number. The Weyl-van der Corput inequality
(Lemma \ref{lem:A}) yields
\begin{align*}
X^{-2\varepsilon}|\cS|^2&\ll \frac{NM}{Q}\sum_{|q|<Q}\Big(1-\frac{q}{Q}\Big)
\sum_{n\sim N}b_{n+q}\overline{b_n}\mathop{\sum_{m\sim M}}_{mn\sim X}
\mathbf{e}\Bigg(\sum_{1\leqslant j\leqslant d}\frac{E_jm^{\beta_j}((n+q)^{\gamma_j}-n^{\gamma_j})}
{M^{\beta_j}N^{\gamma_j}}\Bigg)\nonumber\\
&\ll\frac{X^{2+2\varepsilon}}{Q}+\frac{X^{1+2\varepsilon}}{Q}\sum_{|q|<Q}\sum_{n\sim N}
\Bigg|\mathop{\sum_{m\sim M}}_{mn\sim X}
\mathbf{e}\Bigg(\sum_{1\leqslant j\leqslant d}\frac{E_j((n+q)^{\gamma_j}-n^{\gamma_j})}
{M^{\beta_j}N^{\gamma_j}}m^{\beta_j}\Bigg)\Bigg|.
\end{align*}

Now we estimate the innermost sum. Recall that
$$
\Delta=\Bigg|\sum_{1\leqslant j\leqslant d}\beta_j\cdots(\beta_j-s+1)E_j\Bigg|.
$$
For a fixed integer $n\sim N$, denote
\begin{align*}
f(m)=\sum_{1\leqslant j\leqslant d}\frac{E_j((n+q)^{\gamma_j}-n^{\gamma_j})}
{M^{\beta_j}N^{\gamma_j}}m^{\beta_j}.
\end{align*}
Since
$$
f^{(s)}(M)\asymp
\Bigg|\sum_{1\leqslant j\leqslant d}\beta_j\cdots(\beta_j-s+1)
\frac{E_j((n+q)^{\gamma_j}-n^{\gamma_j})}{M^{\beta_j}N^{\gamma_j}}M^{\beta_j-1}\Bigg|
\asymp \Delta|q|X^{-1}.
$$
In view of Lemma \ref{lem: exponent pair method} we derive that
\begin{align*}
\mathop{\sum_{m\sim M}}_{mn\sim X}\mathbf{e}(f(m))
\ll (\Delta|q|X^{-1})^{\kappa}M^{\lambda}+(\Delta|q|X^{-1})^{-1}.
\end{align*}
Therefore,
\begin{align*}
X^{-5\varepsilon}|\cS|^2\ll X^2Q^{-1}+\Delta^{-1}X^3M^{-1}Q^{-1}+
\Delta^{\kappa}X^{2-\kappa}M^{\lambda-1}Q^{\kappa}.
\end{align*}

This bound is trivial when $Q<1$, so we may consider the case that $Q\in(0,N]$.
Lemma \ref{lemmea: balance} yields
\begin{align*}
X^{-3\varepsilon}\cS\ll M^{1/2}X^{1/2}&+\Delta^{-1/2}X
+M^{(\lambda-\kappa-1)/(2\kappa+2)}X\\
&+\Delta^{\kappa/(2\kappa+2)}M^{(\lambda-1)/(2\kappa+2)}X^{(2+\kappa)/(2\kappa+2)}.
\end{align*}
This completes the proof of Theorem \ref{theorem: 1}.

\section{A lattice method}\label{subsection: polar Lattice}

We review a few terms related to lattice theory. Let $E$ be a subspace of $\RR^n$ having dimension $t$ with $1\leqslant t\leqslant n$. A lattice in $E$ is a subgroup $\Lambda$ of $E$ of the form
$$
\Lambda = \{ m_1z_1 + m_2z_2 + \cdots + m_tz_t : m_i\in \ZZ \},
$$
where $z_1,\cdots,z_t$ are given linearly independent points in $E$. We also say $z_1,\cdots,z_t$ is a basis of $\Lambda$ or generates $\Lambda$. The fundamental parallelepiped (or fundamental domain) for $\Lambda$ corresponding to this basis is the set
$$
\cF = \{ t_1z_1 + t_2z_2 + \cdots + t_nz_n : 0\leqslant t_i < 1 \}.
$$
The $n$-dimensional volume of $\cF$ is called the determinant of $L$, denoted by $\det(\Lambda)$. The polar lattice to $\Lambda$ is the lattice $\Pi$ of points $\pi$ satisfying
$$
\pi \lambda \in \ZZ, \quad \forall \lambda \in \Lambda.
$$


We also need the following lemma which is \cite[Lemma~2.4]{Baker1986}.

\begin{lemma}
	\label{lem:fourier analysis}
	Given any interval $I=[a,b]$, let $\chi_I(x)$ be the indicator function of $I$. There are continuous functions $G_1(x), G_2(x)$ in $L^1(\RR)$ such that
	\begin{align*}
		G_1(x) \leqslant \chi_I(x) \leqslant G_2(x), \\
		\hat{G}_1(t)=\hat{G}_2.(t)=0 \quad \text{for}\ |t|\geqslant 1,
	\end{align*}
	and
	\begin{align*}
		\int_{-\infty}^{+\infty}(\chi_I(x)-G_1(x))\textup{d}x = 1, \quad \int_{-\infty}^{+\infty}(G_2(x)-\chi_I(x))\textup{d}x = 1.
	\end{align*}
\end{lemma}

Now we give the proof of Theorem \ref{theorem: polar}. Let $\chi_d$ denote the indicator function of $\cB_d$. By lemma~\ref{lem:fourier analysis}, we know there exists two continuous functions $\chi^{\pm}$ in $L^1(\RR)$ satisfying
\begin{align*}
	\chi^-(x) \leqslant &\chi_1(x) \leqslant \chi^+(x), \\
	\hat{\chi}^{\pm}(t)=&0 \quad \text{for}\ |t|\geqslant w
\end{align*}
and
$$
\int_{-\infty}^{+\infty}(\chi_1(x)-\chi^-(x))\textup{d}x
=   \int_{-\infty}^{+\infty}(\chi^+(x)-\chi_1(x))\textup{d}x = \frac{1}{w}.
$$
Here and below, the symbol $\hat{g}$ denotes the Fourier transform defined on
$L^{1}(\mathbb{R}^d)$ as
$$
\hat{g}(\boldsymbol{t})=\int_{\mathbb{R}^d}g(\boldsymbol{x})\mathbf{e}(-\boldsymbol{xt})
\textup{d}\boldsymbol{x},
$$
and the parameter $w$ in $(1,+\infty)$ is at our disposal.

Let
$$
\sigma(j,k) = \left\{
\begin{aligned}
	- \quad \text{if}\ j=k, \\
	+ \quad \text{if}\ j\neq k,
\end{aligned}\right. \quad \text{and} \quad
\tau(j,k) = \left\{
\begin{aligned}
	0 \quad \text{if}\ j=k, \\
	+ \quad \text{if}\ j\neq k.
\end{aligned}\right.
$$
We claim that
\begin{align}
	\label{1.1}
	\chi_d(\boldsymbol{x}) \geqslant \sum_{k=1}^{d}\prod_{j=1}^{d}\chi^{\tau(j,k)}(x_j) - (d-1)\prod_{j=1}^{d}\chi^+(x_j),
\end{align}
here $\chi^0(x_j)$ means $\chi_1(x_j)$. To see this, first note that the right-hand side of \eqref{1.1} is clearly non-positive if one of $\chi^0(x_j)$ is zero. If every $\chi^0(x_j)$ is $1$ then the right-hand side can be shown not to exceed $1$ by induction. Hence, it follows that
$$
\chi_d(\boldsymbol{x}) \geqslant \sum_{k=1}^{d}\prod_{j=1}^{d}\chi^{\sigma(j,k)}(x_j) - (d-1)\prod_{j=1}^{d}\chi^+(x_j).
$$
Put
$$
g(\boldsymbol{x}) = \sum_{k=1}^{d}\prod_{j=1}^{d}\chi^{\sigma(j,k)}(x_j) - (d-1)\prod_{j=1}^{d}\chi^+(x_j),
$$
It is obvious that
$$
\hat{g}(\boldsymbol{t}) = 0, \quad \text{when}\ |\boldsymbol{t}|\geqslant w.
$$
Also we obtain that
\begin{align*}
	\hat{g}(\boldsymbol{0}),&=d\int_{\RR^d} \bigg( \prod_{j=1}^{d}\chi^{\sigma(j,k)}(x_j) - \chi_1(x_j) + \chi_1(x_j) \bigg)
	\textup{d}\boldsymbol{x} \\
	&\qquad - (d-1)\int_{\RR^d} \bigg( \prod_{j=1}^{d}\chi^{+}(x_j) - \chi_1(x_j) + \chi_1(x_j) \bigg)
	\textup{d}\boldsymbol{x} \\
	&=d\big(2+\frac{1}{w}\big)^{d-1}\big(2-\frac{1}{w}\big) - (d-1)\big( 2+\frac{1}{w} \big)^{d} = 2^d + O(w^{-1}).
\end{align*}

Let $T(\boldsymbol{x})$ be the function
$$
T(\boldsymbol{x}) = \sum_{\boldsymbol{y} \in \Gamma} g(\boldsymbol{x}+\boldsymbol{y}).
$$
We have the Fourier expansion
$$
T(\boldsymbol{x}) = \sum_{\substack{\boldsymbol{q}\in\Pi \\ |\boldsymbol{q}|<w}}
c_{\boldsymbol{q}}\mathbf{e}(\boldsymbol{qx})
$$
with
$$
c_{\boldsymbol{q}}= \frac{1}{\det \Gamma} \int_{\cF} T(\boldsymbol{x})\mathbf{e}(\boldsymbol{qx})
\textup{d}\boldsymbol{x},
$$
where $\cF$ denotes the fundamental domain for $\Gamma$. By the definition of $T(x)$, it follows that
\begin{align*}
	c_{\boldsymbol{q}}
	&= \frac{1}{\det \Gamma} \sum_{\boldsymbol{y} \in \Gamma}
	\int_{\cF} g(\boldsymbol{x}+\boldsymbol{y}) \mathbf{e}(-\boldsymbol{qx})\textup{d}\boldsymbol{x} \\
	&= \frac{1}{\det \Gamma} \sum_{\boldsymbol{y} \in \Gamma} \int_{\cF + \boldsymbol{y}}
	g(\boldsymbol{x}) \mathbf{e}(-\boldsymbol{qx} + \boldsymbol{qy})\textup{d}\boldsymbol{x}.
\end{align*}
Since $\boldsymbol{q}\in\Pi$ and $\Pi$ is the polar lattice of $\Gamma$, we have $\boldsymbol{qy}\in\ZZ$. Noting that $\cF+\boldsymbol{y}$ is a translation of $\cF$, we achieve that
$$
c_{\boldsymbol{q}} = \frac{1}{\det \Gamma} \int_{\RR^d} g(\boldsymbol{x})
\mathbf{e}(-\boldsymbol{qx})\textup{d}\boldsymbol{x}
=  \frac{1}{\det \Gamma} \hat{g}(\boldsymbol{x}).
$$

Next, we continue to prove this theorem for the case that $a_n\in\mathbb{R}$.
In this case, we can suppose that
$$
\mathcal{A}_1=\{1\leqslant n\leqslant N: a_n\geqslant 0\},\ \ \
\mathcal{A}_2=\{1\leqslant n\leqslant N: a_n<0\}.
$$
One knows that $\alpha_n \in \cB_d \pmod{\Gamma}$ if and only if there exists a
$\boldsymbol{y} \in \Gamma$ such that
$$
\chi_d(\boldsymbol{\alpha}_n + \boldsymbol{y}) = 1.
$$
Then
\begin{align}\label{eq: A1+A2}
	\mathcal{N}=\sum_{n\in\mathcal{A}_1}a_n\sum_{\boldsymbol{y}\in\Gamma}
	\chi_d(\boldsymbol{\alpha}_n+\boldsymbol{y})
	+\sum_{n\in\mathcal{A}_2}a_n\sum_{\boldsymbol{y}\in\Gamma}
	\chi_d(\boldsymbol{\alpha}_n+\boldsymbol{y}).
\end{align}
By the definition of $T(\boldsymbol{x})$ we have
\begin{align*}
	\sum_{n\in\mathcal{A}_1}a_n\sum_{\boldsymbol{y}\in\Gamma}
	\chi_d(\boldsymbol{\alpha}_n+\boldsymbol{y})
	&\geqslant \sum_{n\in\mathcal{A}_1}a_n T(\boldsymbol{\alpha}_n)\\
	&=c_{\boldsymbol{0}} \sum_{n\in\mathcal{A}_1} a_n+
	\sum_{\substack{\boldsymbol{q}\in\Pi \\ 0<|\boldsymbol{q}|<w}} c_{\boldsymbol{q}}
	\sum_{n\in\mathcal{A}_1}a_n \mathbf{e}(\boldsymbol{q}\boldsymbol{\alpha}_n).
\end{align*}
Since the Fourier transform of any product $\prod_{j=1}^d \chi^{\pm}(x_j)$ is bounded in modulus by $(2+1/w)^d$, we have
$$
\frac{|c_{\boldsymbol{p}}|}{c_{\boldsymbol{0}}}
= \frac{\hat{g}(\boldsymbol{q})}{\hat{g}(\boldsymbol{0})}
\leqslant \frac{(2d-1)(2+1/w)^d}{2^d+O(1/w)} \ll 1
$$
for large $w$. Thus we have
\begin{align*}
	\sum_{n\in\mathcal{A}_1}a_n\sum_{\boldsymbol{y}\in\Gamma}
	&\chi_d(\boldsymbol{\alpha}_n+\boldsymbol{y})\\
	&\geqslant  (2^d + O(w^{-1})) \bigg( \sum_{n\in\mathcal{A}_1} a_n
	+ O\Bigg( \sum_{\substack{\boldsymbol{q}\in\Pi \\ 0<|\boldsymbol{q}|<w}}
	\Bigg| \sum_{n\in\mathcal{A}_1}a_n \mathbf{e}(\boldsymbol{q} \boldsymbol{\alpha}_n) \bigg| \Bigg) \Bigg).
\end{align*}
An exactly similar upper bound follows from
$$
\chi_d(x) \leqslant \prod_{j=1}^d \chi^+(x_j),
$$
and we conclude that
\begin{align*}
	\sum_{n\in\mathcal{A}_1}a_n\sum_{\boldsymbol{y}\in\Gamma}
	&\chi_d(\boldsymbol{\alpha}_n+\boldsymbol{y})\\
	&=(2^d + O(w^{-1})) \bigg( \sum_{n\in\mathcal{A}_1} a_n
	+ O\Bigg( \sum_{\substack{\boldsymbol{q}\in\Pi \\ 0<|\boldsymbol{q}|<w}}
	\Bigg| \sum_{n\leqslant N}a_n \mathbf{e}(\boldsymbol{q} \boldsymbol{\alpha}_n) \bigg| \Bigg) \Bigg).
\end{align*}
For the case that $n\in\mathcal{A}_2$,
applying the same method we can also obtain the above asymptotic formula. Therefore it follows from (\ref{eq: A1+A2}) that
\begin{align}\label{eq: cN}
	\cN = \frac{1}{\det \Gamma} \bigg( (2^d + O(w^{-1})) \sum_{n\leqslant N}a_n
	+ O\bigg( \sum_{\substack{\boldsymbol{q}\in\Pi \\ 0<|\boldsymbol{q}|<w}}
	\bigg| \sum_{n\leqslant N}a_n\mathbf{e}(\boldsymbol{q} \boldsymbol{\alpha}_n) \bigg| \bigg) \bigg).
\end{align}
If $f(n)$ are complex weights, we write $a_n=\mu_n+\textup{i}\nu_n$, where $\textup{i}^2=-1$. We have
$$
\mathcal{N}=\sum_{1\leqslant n\leqslant N}\mu_n+\textup{i}\sum_{1\leqslant n\leqslant N}\nu_n.
$$
Continuing with the above process, we also conclude (\ref{eq: cN}).
This completes the proof of Theorem \ref{theorem: polar}.

\section{The intersection of Piatetski-Shapiro sequences}
\label{section: Sums of an arithmetic function}

For a given real $c$ with $1<c<2$, we denote by $\mathbf{1}_{c}(\cdot)$ the characteristic
function of numbers in the Piatetski-Shapiro sequence $\mathcal{N}^c$, i.e.,
$$
\mathbf{1}_c(m)=\left\{\begin{array}{ll} 1,
& \hbox{if\ } m\in\mathcal{N}^c, \\
0, & \hbox{if\ } m\not\in\mathcal{N}^c.
\end{array}\right.
$$
A series of literature (for example, see \cite[Lemma 2]{BackerBBSW2013}) provides the following result.

\begin{lemma} \label{lemma: character of P-S}
Let $c\in(1,2)$ and denote $\gamma=1/c$. Assume that $f$ is an arithmetic
function such that $f(m)\ll m^{\varepsilon}$. Then for any real numbers $1\leqslant M_1<M_2$ we have
\begin{align*}
\sum_{M_1<m\leqslant M_2}f(m)\mathbf{1}_c(m)&
=\gamma\sum_{M_1<m\leqslant M_2}f(m)m^{\gamma-1}\\
&+\sum_{M_1<m\leqslant M_2}f(m)\left(\psi(-(m+1)^\gamma)-\psi(-m^{\gamma})\right)+O(1).
\end{align*}
\end{lemma}

Here we give an analogue result when the variable belongs to the intersection of
Piatetski-Shapiro sequences.

\begin{proposition}\label{theorem: characterization} Let $f$
be an arithmetic function, and let $1<c_1,\cdots,c_d<2$ be real numbers.
Denote $\gamma_j=1/c_j$ and $\sigma_j=1-\gamma_j$ for $j=1,\cdots,d$.
Let $x>0$ be a real number, $\mathcal{L}= \log x$, $1<a<1+\mathcal{L}^{-1}$ and let
$$
\mathcal{M}=\mathcal{M}(c_1,\cdots,c_d,f,x)
=\sum_{x< m\leqslant ax}f(m)\mathbf{1}_{c_1}(m)\cdots \mathbf{1}_{c_d}(m).
$$
We have
\begin{align*}
\mathcal{M}=\gamma_1\cdots \gamma_d&x^{-(\delta_1+\cdots+\delta_d)}\Bigg(
(1+O(\mathcal{L}^{-1}))\sum_{x< m\leqslant ax}f(m)\\
&+O\Bigg(\mathop{\sum_{\boldsymbol{h}\in\mathbb{Z}^d\setminus\{\boldsymbol{0}\}}}
_{|h_j|<\mathcal{L}x^{\delta_j}}\Bigg|\sum_{x< m\leqslant ax}f(m)
\mathbf{e}\left(h_1m^{\gamma_1}+\cdots+h_dm^{\gamma_d}\right)\Bigg|\Bigg)\Bigg),
\end{align*}
where $\boldsymbol{h}=(h_1,\cdots,h_d)$.
\end{proposition}

Such an asymptotic formula is derived from Baker \cite[P. 351]{Baker2014},
who provided the corresponding result when $f(n)=\Lambda(n)$. In this section,
our main purpose is to show that it also holds for any arithmetic
function $f(n)$. We shall show some necessary definitions and properties regarding to our lattice method
in Section \ref{subsection: polar Lattice} and prove Proposition \ref{theorem: characterization}. We also remark that our proof methods are mainly taken from Baker \cite[Section 2]{Baker2014}.

\begin{proof}[Proof of Proposition \ref{theorem: characterization}]
We start by finding a necessary condition and a sufficient condition of the intersection of Piatetski-Shapiro sequences. Then by Theorem~\ref{theorem: polar}, we obtain a lower bound and a upper bound of $\cM$. By showing that the main terms of the lower bound and upper bound are the same, we finish the proof.

We first rewrite the sum $\cM$ as
$$
\cM = \sum_{\substack{x < m \leqslant ax \\ m=\fl{m_1^{c_1 }}=\cdots=\fl{m_d^{c_d}}}} f(m).
$$
Here $m_1,\cdots,m_d$ are arbitrary integers. We have $m=\fl{m_1^{c_1 }}=\cdots=\fl{m_d^{c_d}}$ if and only if
$$
m^{\gamma_j} \leqslant m_j < (m+1)^{\gamma_j}.
$$
In this way, for any $j=1,\cdots,d$, we have
\begin{align}\label{prop6.2.1}
\cM =\sum_{\substack{x < m \leqslant ax \\ m^{\gamma_j} \leqslant m_j < (m+1)^{\gamma_j}}} f(m).
\end{align}

Next we try to find a necessary condition of \eqref{prop6.2.1}. Since $\gamma_j<1$ for any $j=1,\cdots,d$, we note that
$$
m^{\gamma_j} - (m+1)^{\gamma_j} > -\gamma_j x^{\gamma_j-1}.
$$
It follows that
\begin{align}\label{prop6.2.2}
m^{\gamma_j} \leqslant m_j < (m+1)^{\gamma_j} \Rightarrow -\frac{1}{2} \gamma_j x^{\gamma_j-1} \leqslant m^{\gamma_j} + \frac{1}{2} \gamma_j x^{\gamma_j-1} - m_j \leqslant \frac{1}{2} \gamma_j x^{\gamma_j-1}
\end{align}
We apply theorem \ref{theorem: polar} to count sums of the weighted numbers of
$x < m \leqslant ax$ such that
$$
\|z_{m,j}+t_j\|\leqslant v_j\ \ \ (j=1,\cdots,d),
$$
where $\boldsymbol{z}_{m}=(z_{m,1},\cdots,z_{m,d})$ are given points,
$\max\limits_{1\leqslant j\leqslant d}v_j<1/2$
and $t_1,\cdots,t_d$ are fixed. Let
$$
\Gamma=\frac{\mathbb{Z}}{v_1}\times\cdots\times\frac{\mathbb{Z}}{v_d},
\ \ \
\Pi=\left\{(m_1v_1,\cdots,m_dv_d): m_1,\cdots,m_d\in\mathbb{Z}\right\},
$$
and $\boldsymbol{\alpha}_m=((z_{m,1}+t_1)/v_1,\cdots,(z_{m,d}+t_d)/v_d)$. Take
\begin{align*}
 z_{m_{j}} = m^{\gamma_j}, \quad t_j = \frac{1}{2} \gamma_j x^{\gamma_j-1} - m_j, \quad v_j = \frac{1}{2}\gamma_j x^{\gamma_j-1}.
\end{align*}
By \eqref{prop6.2.2}, we apply Theorem \ref{theorem: polar} with the parameter $\omega = \cL/4$ and
$$
\mathcal{N}=\mathop{\sum_{x < m \leqslant ax}}_{\|z_{m,j}+t_j\|\leqslant v_j}f(m),
$$
it follows that
\begin{align*}
\cM \leqslant \cN = v_1\cdots v_d(2^d &+ O(w^{-1}))\sum_{x < m \leqslant ax}f(m)\\
&+O\Bigg(v_1\cdots v_d\mathop{\sum_{\boldsymbol{h}\in\mathbb{Z}^d\setminus\{\boldsymbol{0}\}}}
_{|h_j|<\cL x^{1-\gamma_j}}\Bigg|\sum_{x < m \leqslant ax}f(m)\mathbf{e}(\boldsymbol{h}\boldsymbol{z}_m)
\Bigg|\Bigg).
\end{align*}
Now we give a sufficient condition of \eqref{prop6.2.1} by the same way. Also noting that for a suitable constant $b$, we have
$$
m^{\gamma_j} - (m+1)^{\gamma_j} < -\gamma_j x^{\gamma_j-1} + b\cL x^{\gamma_j-1}.
$$
It follows that
\begin{align}\label{prop6.2.3}
m^{\gamma_j} \leqslant m_j < (m+1)^{\gamma_j} \Leftarrow -\frac{1}{2} \gamma_j x^{\gamma_j-1} + b\cL x^{\gamma_j-1} \leqslant m^{\gamma_j} + \frac{1}{2} \gamma_j x^{\gamma_j-1} - m_j \leqslant \frac{1}{2} \gamma_j x^{\gamma_j-1}.
\end{align}
Using the same definitions of $\Gamma, \Pi,$ and $\boldsymbol{\alpha}_m$ and taking
\begin{align*}
z_{m_{j}} = m^{\gamma_j}, \quad t_j = \frac{1}{2} \gamma_j x^{\gamma_j-1} - m_j, \quad v_j = \frac{1}{2}\gamma_j x^{\gamma_j-1} - b\cL x^{\gamma_j-1},
\end{align*}
it follows from Theorem~\ref{theorem: polar} that
\begin{align*}
\cM \geqslant \cN = v_1\cdots v_d(2^d &+ O(w^{-1}))\sum_{x < m \leqslant ax}f(m)\\
&+O\Bigg(v_1\cdots v_d\mathop{\sum_{\boldsymbol{h}\in\mathbb{Z}^d\setminus\{\boldsymbol{0}\}}}
_{|h_j|<\cL x^{1-\gamma_j}}\Bigg|\sum_{x < m \leqslant ax}f(m)\mathbf{e}(\boldsymbol{h}\boldsymbol{z}_m)
\Bigg|\Bigg).
\end{align*}
Combining the upper bound and the lower bound, we finish the proof.

\end{proof}

\section{Primes in the intersection of Piatetski-Shapiro sequences}
\label{section: intersection P-S primes}

The proof of Theorem \ref{theorem: 1} is carried out
via the the Heath-Brown identity. Here one naturally encounters some bilinear forms
of the shape
$$
\mathfrak{S}=\mathop{\sum_{0<|h_j|<X^{\delta_j}\mathcal{L}}}_{j=1,\cdots,d}
\Bigg|\mathop{{\sum_{m \sim M} \sum_{n \sim N}}}_{mn \sim X} a_m b_n \e \left(
h_1(mn)^{\gamma_1}+\cdots+h_d(mn)^{\gamma_d}\right)\Bigg|,
$$
where $\mathcal{L}:=\log X$,
$N \ge 1, M \ge 1, MN \asymp X$, $|a_m|\ll X^{\varepsilon}$ and $|b_n|\ll X^{\varepsilon}$
for any $\varepsilon>0$. If $b_n=1$ or $b_n=\log n$, we call it a Type I sum and denote it as $\mathfrak{S}_{I}$; otherwise we call it as a Type II sum and denote it as $\mathfrak{S}_{II}$.
Applying Theorem \ref{theorem: 1} and Lemma \ref{lemma: a lemmma to estimate S} we can prove
the following estimates.

\begin{proposition}\label{theorem: Type I} Assume that $\varepsilon>0$ is a
sufficiently small real number and $M<X^{29/42-2\sigma-8\varepsilon}$. Then we have
$$
\mathfrak{S}_{I}\ll X^{1-\varepsilon}.
$$
\end{proposition}

\begin{proposition}\label{theorem: Type II} Assume that $\varepsilon>0$ is a
sufficiently small real number and $X^{117\sigma/20+30\varepsilon} < M < X^{1-2\sigma-10\varepsilon}$.
Then we have
$$
\mathfrak{S}_{II}\ll X^{1-\varepsilon}.
$$
\end{proposition}

\subsection{Proof of Proposition \ref{theorem: Type I}}

Denote
$$
\cT=\sum_{m\sim M}\mathop{\sum_{n\sim N}}_{mn\sim X}a_mb_n\mathbf{e}\left(h_1(mn)^{\gamma_1}
+\cdots+h_d(mn)^{\gamma_d}\right).
$$
We divide the range $m\sim M$ into the following two integer sets:
$$
\mathcal{I}_1=\left\{m\sim M: \left|h_1X^{\gamma_1}\left(\frac{m}{M}\right)^{\gamma_1}
+\cdots+h_dX^{\gamma_d}\left(\frac{m}{M}\right)^{\gamma_d}\right|<X^{2\sigma+10\varepsilon}\right\}
$$
and
$$
\mathcal{I}_2=\{m\sim M: m\notin\mathcal{I}_1\}.
$$
Then we have
\begin{align}\label{eq: SI=T1+T2}
\mathfrak{S}_{I}=\mathop{\sum_{1<|h_j|<X^{\delta_j}\mathcal{L}}}_{j=1,\cdots,d}\cT_1
+\mathop{\sum_{1<|h_j|<X^{\delta_j}\mathcal{L}}}_{j=1,\cdots,d}\cT_2,
\end{align}
where
\begin{align*}
\cT_1=\mathop{\sum_{m\sim M}}_{m\in\mathcal{I}_1}\mathop{\sum_{n\sim N}}_{mn\sim X}a_mb_n
\mathbf{e}\left(h_1(mn)^{\gamma_1}+\cdots+h_d(mn)^{\gamma_d}\right),
\end{align*}
and
\begin{align*}
\cT_2=\mathop{\sum_{m\sim M}}_{m\in\mathcal{I}_2}\mathop{\sum_{n\sim N}}_{mn\sim X}a_mb_n
\mathbf{e}\left(h_1(mn)^{\gamma_1}+\cdots+h_d(mn)^{\gamma_d}\right).
\end{align*}

We begin with $\cT_1$. Clearly,
\begin{align*}
\cT_1\ll X^{2\varepsilon}N\mathop{\sum_{m\sim M}}_{m\in\mathcal{I}_1}1
\ll X^{2\varepsilon}N\#\mathcal{I}_1,
\end{align*}
where $\#\mathcal{I}_1$ is the cardinality of $\mathcal{I}_1$.
It follows from Lemma \ref{lemma: zhai's lemma} that
\begin{align*}
\#\mathcal{I}_1\ll \left(\frac{X^{2\sigma+10\varepsilon}}{|h_1X^{\gamma_1}|+\cdots+
|h_dX^{\gamma_d}|}\right)^{1/(d-1)}M.
\end{align*}
Then we have
\begin{align*}
\mathop{\sum_{1<|h_j|<X^{\delta_j}\mathcal{L}}}_{j=1,\cdots,d}\cT_1
\ll &X^{1+2\sigma/(d-1)+12\varepsilon}\\
&\times\mathop{\sum_{1<|h_j|<X^{\delta_j}\mathcal{L}}}_{j=1,\cdots,d}
\left(\left(|h_1|X^{\gamma_1}\right)^{-1/(d-1)}
+\cdots+\left(|h_d|X^{\gamma_d}\right)^{-1/(d-1)}\right).
\end{align*}
For every integer $1\leqslant j\leqslant d$ we have
$$
\sum_{|h_j|<X^{\delta_j}\log X}\left(|h_j|X^{\gamma_j}\right)^{-1/(d-1)}
\ll X^{\delta_j-1/(d-1)}\mathcal{L}.
$$
Therefore,
\begin{align*}
\mathop{\sum_{1<|h_j|<X^{\delta_j}\mathcal{L}}}_{j=1,\cdots,d}\cT_1
\ll X^{(d+1)\sigma/(d-1)-1/(d-1)+1+13\varepsilon}
\end{align*}
Furthermore, since $\varepsilon>0$ is a sufficiently small real number
and $\sigma<1/(d+1)$, with $\sigma<1/(d+1)-14\varepsilon$ we see that
\begin{align}\label{eq: bound for T1}
\mathop{\sum_{1<|h_j|<X^{\delta_j}\mathcal{L}}}_{j=1,\cdots,d}\cT_1
\ll X^{1-\varepsilon}.
\end{align}

Now we estimate $\cT_2$. Let
$$
\Delta_1=\Bigg|\sum_{1\leqslant j\leqslant d}\gamma_j\cdots(\gamma_j-s+1)
h_jX^{\gamma_j}\Bigg|N^{-1}.
$$
The condition $m\in\mathcal{I}_2$ implies that $\Delta_1\gg X^{2\sigma} N^{-1}$.
Applying a partial summation and
Lemma \ref{lemma: a lemmma to estimate S} with $z_1=h_1m^{\gamma_1},\cdots,z_d=h_dm^{\gamma_d}$
and $\Delta_0=\Delta_1$ we have
\begin{align*}
X^{-2\varepsilon}\cT_2\ll\sum_{m\sim M}
\left(\Delta_1^{\kappa}N^{\lambda}+\Delta_1^{-1}\right)
\ll \Delta_1^{\kappa}X^{\lambda}M^{1-\lambda}+\Delta_1^{-1}XN^{-1}.
\end{align*}
Therefore,
\begin{align}\label{eq: bound for T2}
\mathop{\sum_{1<|h_j|<X^{\delta_j}\mathcal{L}}}_{j=1,\cdots,d}\cT_2
\ll X^{\lambda+\sigma+3\varepsilon}M^{1+\kappa-\lambda}+X^{1-\sigma+3\varepsilon}.
\end{align}

Inserting the exponent pair $(\kappa,\lambda)=(13/84,55/84)$
(see Theorem 6 by Bourgain \cite{Bour}) into (\ref{eq: bound for T2}), and combining
(\ref{eq: SI=T1+T2}) and (\ref{eq: bound for T1}) we derive that
$$
\mathfrak{S}_{I}\ll X^{1-\varepsilon}
$$
provided that $M<X^{29/42-2\sigma-8\varepsilon}$.
This proves Proposition \ref{theorem: Type I}.

\subsection{Proof of Proposition \ref{theorem: Type II}}

For convenience, we continue to use the notation
$$
\cT_2=\mathop{\sum_{m\sim M}}_{m\in\mathcal{I}_2}
\mathop{\sum_{n\sim N}}_{mn\sim X}a_mb_n\mathbf{e}\left(h_1(mn)^{\gamma_1}
+\cdots+h_d(mn)^{\gamma_d}\right).
$$
Following the proof of Theorem \ref{theorem: Type I} we can derive that
\begin{align*}
\mathfrak{S}_{II}=\mathop{\sum_{1<|h_j|<X^{\delta_j}\mathcal{L}}}_{j=1,\cdots,d}\cT_2
+O\left(X^{1-\varepsilon}\right).
\end{align*}
We also remark that $b_n\ll X^{\varepsilon}$ in this case.

Denote
$$
\Delta_2=\Bigg|\sum_{1\leqslant j\leqslant d}\gamma_j\cdots(\gamma_j-s+1)h_jX^{\gamma_j}\Bigg|.
$$
The condition $n\in \mathcal{I}_2$ implies that $\Delta_2\gg X^{2\sigma+10\varepsilon}$.
Applying Theorem \ref{theorem: 1} with $E_j=X^{\gamma_j}h_j$
and $(\kappa,\lambda)=(11/28,11/21)$ (see Page 58 in \cite{Mont1994}) we have
\begin{align*}
X^{-3\varepsilon}\cT_2\ll M^{1/2}X^{1/2}+\Delta_2^{-1/2}X
+M^{-73/234}X
+\Delta_2^{11/78}M^{-20/117}X^{67/78}.
\end{align*}
Hence,
\begin{align*}
X^{-4\varepsilon}\mathop{\sum_{1<|h_j|<X^{\delta_j}\mathcal{L}}}_{j=1,\cdots,d}\cT_2
\ll M^{1/2}X^{1/2+\sigma}+X^{1-5\varepsilon}+M^{-73/234}X^{1+\sigma}
+M^{-20/117}X^{1+\sigma}.
\end{align*}
Since $M>X^{117\sigma/20+30\varepsilon}$ and $M<X^{1-2\sigma-10\varepsilon}$,
we can infer that
\begin{align*}
\mathop{\sum_{1<|h_j|<X^{\delta_j}\mathcal{L}}}_{j=1,\cdots,d}\cT_4\ll X^{1-\varepsilon}.
\end{align*}
Therefore,
$$
\mathfrak{S}_{II}\ll X^{1-\varepsilon}.
$$
This proves Proposition \ref{theorem: Type II}.

\subsection{Proof of Theorem \ref{theorem: 2}}\label{subsection: proof of intersection}

Assume that $1<a\leqslant 1+\mathcal{L}^{-1}$. We shall show that
\begin{align*}
\mathop{\sum_{X< m\leqslant aX}}_{m=\lfloor n_1^{c_1}\rfloor=\cdots=\lfloor n_d^{c_d}\rfloor}\Lambda(m)
=\big(1+O(\mathcal{L}^{-1})\big)\gamma_1\cdots\gamma_d X^{-\sigma}
\sum_{X<m\leqslant aX}\Lambda(m).
\end{align*}
Replacing $X$ by $x\in(X,2X]$ and summing over all disjoint subintervals of $(X,2X]$
of length $\ll X\mathcal{L}^{-1}$, we obtain
$$
\mathop{\sum_{X< m\leqslant 2X}}_{m=\lfloor n_1^{c_1}\rfloor=\cdots=\lfloor n_d^{c_d}\rfloor}\Lambda(m)
=\big(1+O(\mathcal{L}^{-1})\big)\frac{\gamma_1\cdots\gamma_d}{1-\sigma}
(2^{1-\sigma}-1)X^{1-\sigma}.
$$
This yields Theorem 1.2 by standard arguments.

Write the vectors $\boldsymbol{h}=(h_1,\cdots,h_d)$ and $\boldsymbol{0}=(0,\cdots,0)$.
In view of Theorem \ref{theorem: characterization}, it remains to prove that
\begin{align*}
\mathop{\mathop{\sum_{\boldsymbol{h}\in\mathbb{Z}^d\setminus\{\boldsymbol{0}\}}
}_{|h_j|<X^{\delta_j}\mathcal{L}}}_{j=1,\cdots,d}
\Bigg|\sum_{X<m\leqslant aX}\Lambda(m)
\mathbf{e}\left(h_1m^{\gamma_1}+\cdots+h_dm^{\gamma_d}\right)\Bigg|
	\ll X^{1-\varepsilon}
\end{align*}
for any small real $\varepsilon>0$.
In fact, we only need to consider the function $h_1m^{\gamma_1}+\cdots+h_dm^{\gamma_d}$
with non-zero coefficients. It suffices to show that
\begin{align}\label{eq: main goal}
\mathop{\sum_{0<|h_j|<X^{\delta_j}\mathcal{L}}}_{j=1,\cdots,r}
\Bigg|\sum_{X < m \leqslant aX}\Lambda(m)
\mathbf{e}\left(h_1m^{\gamma_1}+\cdots+h_rm^{\gamma_r}\right)\Bigg|
\ll X^{1-\varepsilon},
\end{align}
where $1\leqslant r\leqslant d$.

We apply the Heath-Brown identity (Lemma \ref{lem: Heath-Brown inentity})
with $k=4$. One can find that the left side of (\ref{eq: main goal}) can be written as
linear combination of $O\left((\log X)^8\right)$ sums, each of which is of the shape
\begin{align*}
\Omega(X)=\sum_{n_1\sim N_1}\cdots\sum_{n_8\sim N_8}
&\log n_1\mu(n_5)\mu(n_6)\mu(n_7)\mu(n_8)\\
&\times\mathbf{e}\left(m_1(n_1\cdots n_8)^{\gamma_1}+\cdots+m_r(n_1\cdots n_8)^{\gamma_r}\right),
\end{align*}
where
$$
N_1\cdots N_8\asymp X,\ \ \ \ \ \ N_5,N_6,N_7,N_8\leqslant X^{1/4},
$$
and some $n_i$ may only take value 1. This can be estimated according to
the following three cases.

\textbf{Case I}.  If there exists a $N_j$ with $N_j\geqslant X^{13/42+2\sigma+8\varepsilon}$,
we must have $N_j>X^{1/4}$ and then $j\in\{1,2,3,4\}$. Let
$$
M=\frac{N_1\cdots N_8}{N_j},\ \ \ N=N_j,\ \ \
m=\frac{n_1\cdots n_8}{n_j},\ \ \ n=n_j.
$$
In this case, one can infer that $\Omega(X)$ is a sum of Type I and $M<X^{29/42-2\sigma-8\varepsilon}$.
From Proposition \ref{theorem: Type I} we have
\begin{align*}
\Omega(X)\ll X^{1-\varepsilon}.
\end{align*}

\textbf{Case II}. If there exists a $N_j$ such that
$X^{2\sigma+10\varepsilon}\leqslant N_j<X^{13/42+2\sigma+8\varepsilon}$, we let
$$
N=N_j,\ \ \ M=\frac{N_1\cdots N_8}{N_j},\ \ \ n=n_j,\ \ \ m=\frac{n_1\cdots n_8}{n_j}.
$$
In this case, one can see that $\Omega(X)$ is a sum of Type II and
$$
X^{117\sigma/20+30\varepsilon}<
X^{29/42-2\sigma-8\varepsilon}<M\leqslant X^{1-2\sigma-10\varepsilon},
$$
where we use the fact that $\sigma< \frac{290}{3297}$. By Proposition \ref{theorem: Type II}
we obtain
\begin{align*}
\Omega(X)\ll X^{1-\varepsilon}.
\end{align*}

\textbf{Case III}. If $N_1,\cdots,N_8<X^{2\sigma+10\varepsilon}$, without of generality,
we suppose that $N_1\geqslant N_2\geqslant \cdots\geqslant N_8$ and assume that
$j_0$ is the integer with
$$
N_1N_2\cdots N_{j_0-1}<X^{2\sigma+10\varepsilon}\ \ \ \hbox{and }\ \ \
N_1N_2\cdots N_{j_0}\geqslant X^{2\sigma+10\varepsilon}.
$$
Then we have $j_0\in\{2,\cdots 7\}$ and
$$
X^{2\sigma+10\varepsilon}\leqslant N_1N_2\cdots N_{j_0}=(N_1\cdots N_{j_0-1})N_{j_0}
<\left(X^{2\sigma+10\varepsilon}\right)^2=X^{4\sigma+20\varepsilon}.
$$
Let
$$
N=N_1\cdots N_{j_0},\ \ \ M=N_{j_0+1}\cdots N_8,\ \ \
n=n_1\cdots n_{j_0},\ \ \ m=n_{j_0+1}\cdots n_8.
$$
It follows from $\sigma<\frac{290}{3297}$ that
$$
X^{117\sigma/20+30\varepsilon}<X^{1-4\sigma-20\varepsilon}< M<X^{1-2\sigma-10\varepsilon}.
$$
In this case, $\Omega(X)$ is a sum of Type II. Applying Proposition \ref{theorem: Type II} we have
\begin{align*}
\Omega(X)\ll X^{1-\varepsilon}.
\end{align*}

Now combining the above three cases we obtain (\ref{eq: main goal}).
This completes the proof of Theorem \ref{theorem: 1}.

\section{The iterated Piatetski-Shapiro sequence}
\label{section: proof 3}
In this section, We shall use Theorem~\ref{theorem: polar} to give an initial construction. By this way, we change this question into an estimate of a special exponential sum which contains a floor function in the exponent. Thanks to Kolesnik \cite{Kolesnik2003}, we apply lemma~\ref{lemma: Kolesnik} to solve the problem.

\subsection{Initial construction}

Recall that
$$
\cP(X; c_1,c_2) = \#\left\{p\leqslant X: p=\lfloor\lfloor h^{c_1}\rfloor ^{c_2}\rfloor
\ \hbox{for some integer}\ h\geqslant 1\right\}.
$$
Write $\cL = \log X$ and $1<a<1+\cL^{-1}$. Similar with Section~\ref{subsection: proof of intersection}, we suffice to estimate
\begin{align}
\cM =\sum_{\substack{X < m \leqslant aX \\ m=\lfloor\lfloor h^{c_1}\rfloor ^{c_2}\rfloor }}\Lambda(m)
\end{align}
With the following arguments we find a necessary condition and a sufficient condition to count the integers $m=\lfloor\lfloor h^{c_1}\rfloor ^{c_2}\rfloor$. Similar to the proof of Proposition \ref{theorem: characterization}, by Theorem~\ref{theorem: polar}, but a more complex calculation,\ we obtain a lower bound and a upper bound of
$\cP(X; c_1,c_2)$. By showing that the main terms of the lower bound and upper bound are same, we turn to estimate the error term.

Let $k=\lfloor h^{c_1}\rfloor$. We first have the following relations
$$
m=\lfloor\lfloor h^{c_1}\rfloor ^{c_2}\rfloor \Leftrightarrow m^{\gamma_2} \leqslant k < (m+1)^{\gamma_2} \quad \text{and} \quad k^{\gamma_1} \leqslant h < (k+1)^{\gamma_1}.
$$
Since there is an integer between $m^{\gamma_2}$ and $(m+1)^{\gamma_2}$ for our $m$ (which has at most one integer), we note that $k = \lfloor (m+1)^{\gamma_2}\rfloor$. We derive that
$$
\lfloor (m+1)^{\gamma_2}\rfloor^{\gamma_1} \leqslant h < (\lfloor (m+1)^{\gamma_2}\rfloor+1)^{\gamma_1}.
$$
Hence we obtain that
\begin{align}
	\label{7.1.1}
\cM = \sum_{\substack{X < m \leqslant aX \\ m^{\gamma_2} \leqslant k < (m+1)^{\gamma_2} \\ \lfloor (m+1)^{\gamma_2}\rfloor^{\gamma_1} \leqslant h < (\lfloor (m+1)^{\gamma_2}\rfloor+1)^{\gamma_1}}}\Lambda(m),
\end{align}
where $k$ and $h$ are defined as above.

Next we try to find a necessary condition of \eqref{7.1.1}. It is easy to see that
\begin{align}
	\label{7.1.2}
m^{\gamma_2} \leqslant k < (m+1)^{\gamma_2} \Rightarrow -\frac{1}{2}\gamma_2 X^{\gamma_2-1} \leqslant m^{\gamma_2}+\frac{1}{2}\gamma_2 X^{\gamma_2-1} -k \leqslant \frac{1}{2}\gamma_2 X^{\gamma_2-1}
\end{align}
and
\begin{align}
	\label{7.1.3}
\nonumber \lfloor &(m+1)^{\gamma_2}\rfloor^{\gamma_1} \leqslant h < (\lfloor (m+1)^{\gamma_2}\rfloor+1)^{\gamma_1} \\ &\Rightarrow -\frac{1}{2}\gamma_1 X^{\gamma_2(\gamma_1-1)} \leqslant\lfloor (m+1)^{\gamma_2}\rfloor^{\gamma_1}+\frac{1}{2}\gamma_1 X^{\gamma_2(\gamma_1-1)} -h \leqslant\frac{1}{2}\gamma_1 X^{\gamma_2(\gamma_1-1)}.
\end{align}
Following the arguments of Proposition \ref{theorem: characterization}, we let
\begin{align*}
&z_{n_1} = \lfloor(m+1)^{\gamma_2}\rfloor^{\gamma_1}, \quad &&z_{n_2} = m^{\gamma_2}, \\
&t_1 = \frac{1}{2}\gamma_1 X^{\gamma_2(\gamma_1-1)} -h, \quad &&t_2 = \frac{1}{2}\gamma_2 X^{\gamma_2-1} -k, \\
&v_1 = \frac{1}{2}\gamma_1 X^{\gamma_2(\gamma_1-1)}, \quad &&v_2 = \frac{1}{2}\gamma_2 X^{\gamma_2-1}.
\end{align*}
In this way, by \eqref{7.1.2} and \eqref{7.1.3} we obtain that
\begin{align}
	\label{7.1.4}
\nonumber \cM &\leqslant (1+o(1)) \sum_{\substack{X < m \leqslant aX \\ \| z_{n_j} + t_j \| \leqslant v_j \ (j=1,2)}}\Lambda(m) \\
\nonumber &= \gamma_1\gamma_2 X^{\gamma_1\gamma_2-1} \bigg( (1+o(1)) \sum_{X < m \leqslant aX}\Lambda(m) \\
&\quad+O\Big( \sum_{\substack{(k,h) \in \ZZ^d \backslash (0,0) \\ |k|<\cL X^{1-\gamma_2} \\ |h|<\cL X^{\gamma_2(1-\gamma_1)}}} \Big| \Lambda(m) \e (km^{\gamma_2} + h\lfloor(m+1)^{\gamma_2}\rfloor^{\gamma_1}) \Big| \Big) \bigg).
\end{align}

Now we try to give a lower bound of $\cM$. Since for a suitable constant $b_1$, we have that $ (\lfloor (m+1)^{\gamma_2}\rfloor+1)^{\gamma_1} $ lies in the following interval
\begin{align*}
\bigg(\lfloor (m+1)^{\gamma_2}\rfloor^{\gamma_1} + \gamma_1X^{\gamma_2(\gamma_1-1)} - b_1\cL^{-1}X^{\gamma_2(\gamma_1-1)},\lfloor (m+1)^{\gamma_2}\rfloor^{\gamma_1} + \gamma_1X^{\gamma_2(\gamma_1-1)}\bigg].
\end{align*}
If we replace $-\frac{1}{2}\gamma_1 X^{\gamma_2(\gamma_1-1)}$ by $-\frac{1}{2}\gamma_1 X^{\gamma_2(\gamma_1-1)} + b_1\cL^{-1}X^{\gamma_2(\gamma_1-1)}$ in the left-hand inequality in \eqref{7.1.3}, we find that
\begin{align}
	\label{7.1.5}
\nonumber -\frac{1}{2}\gamma_1 X^{\gamma_2(\gamma_1-1)} + b_1\cL^{-1}&X^{\gamma_2(\gamma_1-1)} \leqslant\lfloor (m+1)^{\gamma_2}\rfloor^{\gamma_1}+\frac{1}{2}\gamma_1 X^{\gamma_2(\gamma_1-1)} -h \leqslant\frac{1}{2}\gamma_1 X^{\gamma_2(\gamma_1-1)} \\
&\Rightarrow \lfloor (m+1)^{\gamma_2}\rfloor^{\gamma_1} \leqslant h < (\lfloor (m+1)^{\gamma_2}\rfloor+1)^{\gamma_1}.
\end{align}
Similarly, there is also a suitable constant $b_2$, such that
\begin{align}
	\label{7.1.6}
\nonumber -\frac{1}{2}\gamma_2 X^{\gamma_2-1} + b_2\cL^{-1}&X^{\gamma_2-1} \leqslant m^{\gamma_2}+\frac{1}{2}\gamma_2 X^{\gamma_2-1} -k \leqslant \frac{1}{2}\gamma_2 X^{\gamma_2-1}\\
&\Rightarrow m^{\gamma_2} \leqslant k < (m+1)^{\gamma_2}.
\end{align}
With the same definition of $z_{n_j}$ and $t_j$ of the lower bound and
\begin{align*}
v_1 = \frac{1}{2}\gamma_1 X^{\gamma_2(\gamma_1-1)} - b_1\cL^{-1}&X^{\gamma_2(\gamma_1-1)}, \quad v_2 = \frac{1}{2}\gamma_2 X^{\gamma_2-1} - b_2\cL^{-1}X^{\gamma_2-1},i
\end{align*}
by \eqref{7.1.5} and \eqref{7.1.6} we achieve that
\begin{align}
\label{7.1.7}
\cM &\geqslant \gamma_1\gamma_2 X^{\gamma_1\gamma_2-1} \bigg( (1+o(1)) \sum_{X < m \leqslant aX}\Lambda(m) \nonumber \\
&\quad+O\Big( \sum_{\substack{(k,h) \in \ZZ^2 \backslash (0,0) \\ |k|<\cL X^{1-\gamma_2} \\ |h|<\cL X^{\gamma_2(1-\gamma_1)}}} \Big| \sum_{X < m \leqslant aX} \Lambda(m) \e (km^{\gamma_2} + h\lfloor(m+1)^{\gamma_2}\rfloor^{\gamma_1}) \Big| \Big) \bigg).
\end{align}
Combining \eqref{7.1.4} and \eqref{7.1.7}, our main goal is to prove
\begin{align}\label{eq: main goal in intersection ps}
\sum_{\substack{(k,h) \in \ZZ^2 \backslash (0,0) \\ |k|<\cL X^{1-\gamma_2} \\ |h|<\cL X^{\gamma_2(1-\gamma_1)}}} \Big| \sum_{X < m \leqslant aX} \Lambda(m) \e (km^{\gamma_2} + h\lfloor(m+1)^{\gamma_2}\rfloor^{\gamma_1})\Big|\ll X^{1-\varepsilon_0}
\end{align}
for some real $\varepsilon_0>0$.

\subsection{Application of Kolesnik's result}

For convenience, we denote
$$
\cF=\sum_{X < m \leqslant aX}\Lambda(m)
\e (km^{\gamma_2} + h\lfloor(m+1)^{\gamma_2}\rfloor^{\gamma_1}).
$$
Let $\varepsilon>0$ be a sufficiently small real number, and let
$$
r=\fl{\varepsilon^{-1}}+1\ \ \ \hbox{and}\ \ \ W=X^{1-\gamma_1\gamma_2+\varepsilon}.
$$
We apply Lemma \ref{lemma: Kolesnik} with
$$
a(x)=\Lambda(x),\ \ \ g(x)=x,\ \ \ h(x)=(x+1)^{\gamma_2}
$$
and
$$
f(\mu,\nu)=k\mu^{\gamma_2}+h\left((\mu+1)^{\gamma_2}-\nu\right)^{\gamma_1}.
$$
Since $\mu=g(m)$ and $\nu=h(m)$, we can see that
\begin{align*}
\frac{\partial f(\mu,\nu)}{\partial \mu}
&=h\gamma_1\gamma_2\left((\mu+1)^{\gamma_2}-\nu\right)^{\gamma_1-1}
(\mu+1)^{\gamma_2-1}+\gamma_2k\mu^{\gamma_2-1}\\
&\ll |k|X^{\gamma_2-1}\ll \log X.
\end{align*}
Thus we can take $\lambda=O\left(\log X\right)$. It follows that
\begin{align*}
\cF=\cF_1+O(\cF_2)+O(\cF_3),
\end{align*}
where
\begin{align*}
\cF_1=\frac{1}{W}&\sum_{w=0}^{W}\sum_{j=0}^{\infty}\left(\frac{\sin(\pi j/M)}{\pi j/W}\right)^{r+1}
\mathbf{e}\left(-\frac{(2w+1)j}{2W}\right)\\
&\times\sum_{X < m \leqslant aX}
\Lambda(m)\mathbf{e}\left(h\left((m+1)^{\gamma_2}-\frac{w}{W}\right)^{\gamma_1}
+km^{\gamma_2}+j(m+1)^{\gamma_2}\right),
\end{align*}
\begin{align*}
\cF_2=\frac{\lambda r+r}{W}\sum_{X\leqslant m<a X}\Lambda(m),
\end{align*}
and
\begin{align*}
\cF_3=\frac{r}{W}\sum_{j=0}^{\infty}\frac{\sin(2\pi rj/W)}{\sin(\pi j/W)}
\left(\frac{\sin(\pi j/W)}{\pi j/W}\right)^{r+1}
\sum_{X < m \leqslant aX}\Lambda(m)\mathbf{e}\left(j (m+1)^{\gamma_2}\right).
\end{align*}

Using similar methods with Section~\ref{section: intersection P-S primes}, we have the following three propositions and we will give the detailed proof in the next section.

\begin{proposition}\label{F1}
Let $\cF_1$ be defined as above, if $\gamma_1$ and $\gamma_2$ satisfy that
$$
\left\{\begin{array}{ll}
4\gamma_1\gamma_2+\gamma_2-4>0, & \\
15\gamma_1\gamma_2-3\gamma_2-11>0, &\\
26\gamma_1\gamma_2-2\gamma_2-23>0, & \\
			15\gamma_1\gamma_2-14>0.
\end{array}\right.
$$
For any sufficiently small $\varepsilon>0$, we have
$$
\sum_{\substack{(k,h) \in \ZZ^2 \backslash (0,0) \\ |k|<\cL X^{1-\gamma_2} \\ |h|<\cL X^{\gamma_2(1-\gamma_1)}}}|\cF_1| \ll X^{1-\varepsilon}.
$$
\end{proposition}

\begin{proposition}\label{F2}
Let $\cF_2$ be defined as above. For any sufficiently small $\varepsilon>0$, we have
$$
\sum_{\substack{(k,h) \in \ZZ^2 \backslash (0,0) \\ |k|<\cL X^{1-\gamma_2} \\ |h|<\cL X^{\gamma_2(1-\gamma_1)}}}|\cF_2| \ll X^{1-\varepsilon/2}.
$$
\end{proposition}

\begin{proposition}\label{F3}
Let $\cF_3$ be defined as above, if $\gamma_1$ and $\gamma_2$ satisfy that
$$
\left\{\begin{array}{ll}
	9\gamma_1\gamma_2-3\gamma_2-5>0, & \\
	6\gamma_1\gamma_2-5>0, & \\
	14\gamma_1\gamma_2-2\gamma_2-11>0. &
\end{array}\right.
$$
For any sufficiently small $\varepsilon>0$, we have
$$
\sum_{\substack{(k,h) \in \ZZ^2 \backslash (0,0) \\ |k|<\cL X^{1-\gamma_2} \\ |h|<\cL X^{\gamma_2(1-\gamma_1)}}}|\cF_3| \ll X^{1-\varepsilon/2}.
$$
\end{proposition}

\subsection{Complete the proof of Theorem \ref{theorem: 3}}

Assume that Propositions~\ref{F1}, \ref{F2} and \ref{F3}
hold, i,e.
\begin{align}\label{eq: all conditions}
	\left\{\begin{array}{ll}
		9\gamma_1\gamma_2-3\gamma_2-5>0, & \\
		6\gamma_1\gamma_2-5>0, & \\
		14\gamma_1\gamma_2-2\gamma_2-11>0, & \\
		4\gamma_1\gamma_2+\gamma_2-4>0, & \\
15\gamma_1\gamma_2-3\gamma_2-11>0, &\\
26\gamma_1\gamma_2-2\gamma_2-23>0, & \\
			15\gamma_1\gamma_2-14>0.
	\end{array}\right.
\end{align}
This shows that \eqref{eq: main goal in intersection ps} holds provided that $\gamma_1$ and $\gamma_2$ satisfy \eqref{eq: all conditions}. We immediately obtain
\begin{align*}
	\sum_{\substack{(k,h) \in \ZZ^2 \backslash (0,0) \\ |k|<\cL X^{1-\gamma_2} \\ |h|<\cL X^{\gamma_2(1-\gamma_1)}}} \Big| \sum_{X < m \leqslant aX} \Lambda(m) \e (km^{\gamma_2} + h\lfloor(m+1)^{\gamma_2}\rfloor^{\gamma_1})\Big|\ll X^{1-\varepsilon/3}.
\end{align*}
In fact,
$$
26\gamma_1\gamma_2-2\gamma_2-23<
\left\{\begin{array}{ll}
	9\gamma_1\gamma_2-3\gamma_2-5, & \\
	6\gamma_1\gamma_2-5, & \\
	14\gamma_1\gamma_2-2\gamma_2-11, & \\
	4\gamma_1\gamma_2+\gamma_2-4, & \\
15\gamma_1\gamma_2-3\gamma_2-11, &\\
			15\gamma_1\gamma_2-14. & \\
\end{array}\right.
$$
Therefore, if $26\gamma_1\gamma_2-2\gamma_2-23>0$, we have
$$
\cP(X;c_1,c_2)=(1+o(1))\frac{X^{\gamma_1\gamma_2}}{c_1c_2\log X}.
$$
This completes the proof of Theorem \ref{theorem: 3}.

\section{Proof of Proposition~\ref{F1}, \ref{F2} and \ref{F3} } \label{section: 9}

In this section we will give the proof of proposition~\ref{F1}, \ref{F2} and \ref{F3}.
\subsection{Proof of Proposition~\ref{F1}}
Proposition~\ref{F1} is proved by a technical simplification and then solved by a similar construction of Theorem \ref{theorem: 1}. We divide $\cF_1$ into the following three parts:
\begin{align*}
	\cF_1=\cF_{11}+\cF_{12}+\cF_{13},
\end{align*}
where the summation is taken over $j\leqslant \varepsilon W$ in $\cF_{11}$,
$\varepsilon W<j\leqslant W^{1+2/r}$ in $\cF_{12}$ and $j>W^{1+2/r}$
in $\cF_{13}$. For convenience, we also denote
$$
\Psi:=\sum_{X < m \leqslant aX}
\Lambda(m)\mathbf{e}\Big(h\Big((m+1)^{\gamma_2}-\frac{w}{W}\Big)^{\gamma_1}
+km^{\gamma_2}+j(m+1)^{\gamma_2}\Big),
$$
and
$$
\mathfrak{T}:=\mathop{\sum_{m\sim M}\sum_{n\sim N}}_{X < mn \leqslant aX}
a_mb_n\mathbf{e}\Big(h\Big((mn+1)^{\gamma_2}-\frac{w}{W}\Big)^{\gamma_1}
+k(mn)^{\gamma_2}+j(mn+1)^{\gamma_2}\Big),
$$
where $a_m$ and $b_n$ are real coefficients such that
$a_m\ll X^{\varepsilon}$. If $b_n=1$ or $\log n$ we write it as $\mathfrak{T}_{I}$.
If $b_n\ll X^{\varepsilon}$, we write it as $\mathfrak{T}_{II}$.

We apply the same method as Theorem \ref{theorem: 1} (also see Theorem \ref{theorem: Type I}
and Theorem \ref{theorem: Type II}) to estimate $\mathfrak{T}_{I}$ and $\mathfrak{T}_{II}$. Using
the exponent pair method we obtain the following two lemmas.

\begin{lemma}\label{lemma: Gamma I} Suppose that $M<X^{1/3}$,
$W=X^{1-\gamma_1\gamma_2+\varepsilon}$, $j<W^{1+2\varepsilon}$, and
	\begin{align}\label{eq: condition 2}
		\left\{\begin{array}{ll} 4\gamma_1\gamma_2+\gamma_2-4>0, & \\
			15\gamma_1\gamma_2-3\gamma_2-11>0, &
		\end{array}\right.
	\end{align}
	we have
	\begin{align*}
		\sum_{\substack{(k,h) \in \ZZ^2 \backslash (0,0) \\ |k|<\cL X^{1-\gamma_2} \\ |h|<\cL X^{\gamma_2(1-\gamma_1)}}}
		|\mathfrak{T}_{I}|\ll X^{\gamma_1\gamma_2-2\varepsilon}.
	\end{align*}
\end{lemma}

\begin{proof} For a fixed $m\sim M$, we assume that $x\sim N$ and
	$$
	f(x)=h\Big((mx+1)^{\gamma_2}-\frac{w}{W}\Big)^{\gamma_1}
	+k(mx)^{\gamma_2}+j(mx+1)^{\gamma_2}.
	$$
	Clearly,
	\begin{align*}
		f^{(s)}(x)\asymp \bigg|&\gamma_1\gamma_2(\gamma_1\gamma_2-1)\cdots(\gamma_1\gamma_2-s+1)
		hX^{\gamma_1\gamma_2}\\
		&+\gamma_2(\gamma_2-1)\cdots(\gamma_2-s+1)(j+k)X^{\gamma_2}\bigg|N^{-1}\cdot N^{1-s}.
	\end{align*}
	
	We divide the range $m\sim M$ into the following two parts:
	$$
	\mathcal{J}_1:=\left\{m\sim M: \left|hX^{\gamma_1\gamma_2}\left(\frac{m}{M}\right)
	^{\gamma_1\gamma_2}
	+(k+j)X^{\gamma_2}\left(\frac{m}{M}\right)^{\gamma_2}\right|<
	X^{2-2\gamma_1\gamma_2+6\varepsilon}\right\}
	$$
	and
	$$
	\mathcal{J}_2:=\{m\sim M: m\notin\mathcal{J}_1\}.
	$$
	Then we have
	\begin{align*}
		\mathfrak{T}_{I}=\mathfrak{T}_{I,1}+\mathfrak{T}_{II,2},
	\end{align*}
	where
	$$
	\mathfrak{T}_{I,1}:=\mathop{\sum_{m\sim M}}_{m\in\mathcal{J}_1}
    \mathop{\sum_{n\sim N}}_{X < mn \leqslant aX}
	a_mb_n\mathbf{e}\Big(h\left((mn+1)^{\gamma_2}-\frac{w}{W}\right)^{\gamma_1}
	+k(mn)^{\gamma_2}+j(mn+1)^{\gamma_2}\Big).
	$$
	and
	$$
	\mathfrak{T}_{I,2}:=\mathop{\sum_{m\sim M}}_{m\in\mathcal{J}_2}
    \mathop{\sum_{n\sim N}}_{X < mn \leqslant aX}
	a_mb_n\mathbf{e}\Big(h\left((mn+1)^{\gamma_2}-\frac{w}{W}\right)^{\gamma_1}
	+k(mn)^{\gamma_2}+j(mn+1)^{\gamma_2}\Big).
	$$
	
	We begin with the sum $\mathfrak{T}_{I,1}$. Clearly,
	\begin{align*}
		\mathfrak{T}_{I,1}\ll X^{2\varepsilon}N\mathop{\sum_{m\sim M}}_{m\in\mathcal{J}_1}1
		\ll X^{2\varepsilon}N\#\mathcal{J}_1.
	\end{align*}
	It follows from Lemma \ref{lemma: zhai's lemma} that
	\begin{align*}
		\#\mathcal{J}_1\ll \left(\frac{X^{2-2\gamma_1\gamma_2+6\varepsilon}}{|hX^{\gamma_1\gamma_2}|
			+|(j+k)X^{\gamma_2}|}\right)M.
	\end{align*}
   This can be bounded by considering two cases: $h\neq0$ with $j+k=0$; or $h=0$
   with $j+k\neq0$. By
	\begin{align*}
		4\gamma_1\gamma_2+\gamma_2-4>0
	\end{align*}
	we have
	\begin{align*}
		\sum_{\substack{(k,h) \in \ZZ^2 \backslash (0,0) \\ |k|<\cL X^{1-\gamma_2} \\
        |h|<\cL X^{\gamma_2(1-\gamma_1)}}}
		|\mathfrak{T}_{I,1}|\ll X^{\gamma_1\gamma_2-2\varepsilon}.
	\end{align*}
	
	Now we estimate $\mathfrak{T}_{I,2}$. Let
	$$
	\Delta_3:=\bigg|hX^{\gamma_1\gamma_2}\left(\frac{m}{M}\right)^{\gamma_1\gamma_2}
	+(k+j)X^{\gamma_2}\left(\frac{m}{M}\right)^{\gamma_2}\bigg|N^{-1}.
	$$
	The condition $m\in\mathcal{J}_2$ implies that
    $X^{2-2\gamma_1\gamma_2+6\varepsilon} N^{-1}\ll \Delta_3\ll
	X^{1-\gamma_1\gamma_2+\gamma_2+2\varepsilon} N^{-1}$.
	Applying a partial summation and Lemma \ref{lem: exponent pair method} with $(\kappa,\lambda)=(1/2,1/2)$
	we have
	$$
	\mathop{\sum_{n\sim N}}_{X < mn \leqslant aX}b_n\mathbf{e}(f(n))
	\ll X^{(1-\gamma_1\gamma_2+\gamma_2)/2+2\varepsilon}+X^{-2+2\gamma_1\gamma_2-5\varepsilon} N.
	$$
	Therefore, the assumptions $M<X^{1/3}$ and
	\begin{align*}
		15\gamma_1\gamma_2-3\gamma_2-11>0
	\end{align*}
	imply that
	\begin{align*}
		\sum_{\substack{(k,h) \in \ZZ^2 \backslash (0,0) \\ |k|<\cL X^{1-\gamma_2} \\ |h|<\cL X^{\gamma_2(1-\gamma_1)}}}|\mathfrak{T}_{I,2}|
		\ll X^{\gamma_1\gamma_2-2\varepsilon}.
	\end{align*}

	Therefore,
	\begin{align*}
		\sum_{\substack{(k,h) \in \ZZ^2 \backslash (0,0) \\ |k|<\cL X^{1-\gamma_2} \\ |h|<\cL X^{\gamma_2(1-\gamma_1)}}}
		|\mathfrak{T}_{I}|\ll
\sum_{\substack{(k,h) \in \ZZ^2 \backslash (0,0) \\ |k|<\cL X^{1-\gamma_2} \\ |h|<\cL X^{\gamma_2(1-\gamma_1)}}}
		(|\mathfrak{T}_{I,1}|+|\mathfrak{T}_{I,2}|)\ll
X^{\gamma_1\gamma_2-2\varepsilon}.
	\end{align*}
\end{proof}

\begin{lemma}\label{lemma: Gamma II} Assume that $X^{1/2}<M<X^{2/3}$, $W=X^{1-\gamma_1\gamma_2+\varepsilon}$, $j<W^{1+2\varepsilon}$ and
	\begin{align}\label{eq: condition 3}
		\left\{\begin{array}{ll} 26\gamma_1\gamma_2-2\gamma_2-23>0, & \\
			15\gamma_1\gamma_2-14>0,
		\end{array}\right.
	\end{align}
	we have
	\begin{align*}
		\sum_{\substack{(k,h) \in \ZZ^2 \backslash (0,0) \\ |k|<\cL X^{1-\gamma_2} \\ |h|<\cL X^{\gamma_2(1-\gamma_1)}}}|\mathfrak{T}_{II}|\ll X^{\gamma_1\gamma_2-2\varepsilon}.
	\end{align*}
\end{lemma}

\begin{proof}
	We divide the range $m\sim M$ into the following two parts:
	$$
	\mathcal{J}_3=\bigg\{m\sim M:
	\bigg|hX^{\gamma_1\gamma_2}
	\Big(\frac{m}{M}\Big)^{\gamma_1\gamma_2}
	+(k+j)X^{\gamma_2}\Big(\frac{m}{M}\Big)^{\gamma_2}\bigg|
	<X^{12+\gamma_2-13\gamma_1\gamma_2+24\varepsilon}\bigg\}
	$$
	and
	$$
	\mathcal{J}_4=\{m\sim M: m\notin\mathcal{J}_3\}.
	$$
	Then we have
	\begin{align*}
		\mathfrak{T}_{II}=\mathfrak{T}_{II,1}+\mathfrak{T}_{II,2},
	\end{align*}
	where
	$$
	\mathfrak{T}_{II,1}=\mathop{\sum_{m\sim M}}_{m\in\mathcal{J}_3}\mathop{\sum_{n\sim N}}_{X < mn \leqslant aX}
	a_mb_n\mathbf{e}\Big(h\Big((mn+1)^{\gamma_2}-\frac{w}{W}\Big)^{\gamma_1}
	+k(mn)^{\gamma_2}+j(mn+1)^{\gamma_2}\Big),
	$$
	and
	$$
	\mathfrak{T}_{II,2}=\mathop{\sum_{m\sim M}}_{m\in\mathcal{J}_4}\mathop{\sum_{n\sim N}}_{X < mn \leqslant aX}
	a_mb_n\mathbf{e}\left(h\left((mn+1)^{\gamma_2}-\frac{w}{W}\right)^{\gamma_1}
	+k(mn)^{\gamma_2}+j(mn+1)^{\gamma_2}\right).
	$$
	
	We begin with $\mathfrak{T}_{II,1}$. Clearly,
	\begin{align*}
		\mathfrak{T}_{II,1}\ll X^{2\varepsilon}N\mathop{\sum_{m\sim M}}_{m\in\mathcal{J}_3}1
		\ll X^{2\varepsilon}N\#\mathcal{J}_3.
	\end{align*}
	It follows from Lemma \ref{lemma: zhai's lemma} that
	\begin{align*}
		\#\mathcal{J}_3\ll \bigg(\frac{X^{12+\gamma_2-13\gamma_1\gamma_2+24\varepsilon}}
		{|hX^{\gamma_1\gamma_2}|
			+|(k+j)X^{\gamma_2}|}\bigg)M.
	\end{align*}
This can be bounded by considering two cases: $h\neq0$ with $j+k=0$; or $h=0$
   with $j+k\neq0$. By $15\gamma_1\gamma_2-14>0$ we have
	\begin{align*}
		\sum_{\substack{(k,h) \in \ZZ^2 \backslash (0,0) \\ |k|<\cL X^{1-\gamma_2} \\ |h|<\cL X^{\gamma_2(1-\gamma_1)}}}|\mathfrak{T}_{II,1}|
		\ll X^{\gamma_1\gamma_2-2\varepsilon}.
	\end{align*}
	
	Now we estimate $\mathfrak{T}_{II,2}$.
	Applying the Cauchy inequality we have
	\begin{align*}
		|\mathfrak{T}_{II,2}|^2\leqslant X^{\varepsilon}M\sum_{m\sim M}
		\Bigg|\sum_{n\sim N}b_n\mathbf{e}&\bigg(\left(h(mn+1)^{\gamma_2}-\frac{w}{W}\right)^{\gamma_1}\\
		&+k(mn)^{\gamma_2}+j(mn+1)^{\gamma_2}\bigg)\Bigg|^2.
	\end{align*}
	Let $Q\leqslant N$ be an integer, to be chosen later. The Weyl-van der Corput inequality
   (see Lemma \ref{lem:A}) yields
	\begin{align*}
		|\mathfrak{T}_{II,2}|^2\ll\frac{X^{2+2\varepsilon}}{Q}
		+\frac{X^{1+\varepsilon}}{Q}\sum_{1\leqslant |q|\leqslant Q}
		\sum_{n\sim N}\Bigg|\sum_{m\sim M}\mathbf{e}(f(m))\Bigg|,
	\end{align*}
	where
	\begin{align*}
		f(x)=&h\left((xn+xq+1)^{\gamma_2}-\frac{w}{W}\right)^{\gamma_1}
		+k(xn+xq)^{\gamma_2}+j(xn+xq+1)^{\gamma_2}\\
		&-h\left((xn+1)^{\gamma_2}-\frac{w}{W}\right)^{\gamma_1}
		+k(xn)^{\gamma_2}+j(xn+1)^{\gamma_2}.
	\end{align*}
	For fixed $n\sim N$ and $q$, and for $x\sim M$, it is easy to see that
	\begin{align*}
		f^{(s)}(x)\asymp\Big|&(\gamma_1\gamma_2)(\gamma_1\gamma_2-1)\cdots(\gamma_1\gamma_2-s+1)
		hX^{\gamma_1\gamma_2}\\
		&+\gamma_2(\gamma_2-1)\cdots(\gamma_2-s+1)(j+k)X^{\gamma_2}\Big||q|X^{-1}M^{1-s}.
	\end{align*}
	
	Let
	$$
	\Delta_4=\left|hX^{\gamma_1\gamma_2}
	\left(\frac{m}{M}\right)^{\gamma_1\gamma_2}
	+(k+j)X^{\gamma_2}\left(\frac{m}{M}\right)^{\gamma_2}\right|.
	$$
	The condition $m\in\mathcal{J}_4$ implies that
	$$
   X^{12+\gamma_2-13\gamma_1\gamma_2+24\varepsilon}\ll \Delta_4\ll
   X^{1+\gamma_2-\gamma_1\gamma_2+2\varepsilon}.
   $$
	Applying Lemma \ref{lem: exponent pair method} with $(\kappa,\lambda)=(1/2,1/2)$
	we can see that
	\begin{align*}
		\mathfrak{T}_{II,2}\ll X^{(3-\gamma_1\gamma_2+\gamma_2)/4+\varepsilon}N^{1/4}Q^{1/4}
		+X^{(13\gamma_1\gamma_2-\gamma_2-10)/2-11\varepsilon}N^{1/2}Q^{-1/2}
		+X^{1+\varepsilon}Q^{-1/2}.
	\end{align*}
	This bound is trivial when $Q<1$, so we may restrict that $0<Q\leqslant N$.
	Lemma \ref{lemmea: balance} and $X^{1/3}<N<X^{1/2}$ yield
    \begin{align*}
    \mathfrak{T}_{II,2}\ll
		X^{2\gamma_1\gamma_2-1-3\varepsilon}
       +X^{(11-2\gamma_1\gamma_2+2\gamma_2)/12+\varepsilon}
       +X^{(13\gamma_1\gamma_2-\gamma_2-10)/2-11\varepsilon}
       +X^{5/6+\varepsilon}.
    \end{align*}
    It follows from the inequalities that
    $$
    26\gamma_1\gamma_2-2\gamma_2-23>0\ \ \ \hbox{and}
    \ \ \ 12\gamma_1\gamma_2-11>0
    $$
    we have
	\begin{align*}
		\sum_{\substack{(k,h) \in \ZZ^2 \backslash (0,0) \\ |k|<\cL X^{1-\gamma_2} \\ |h|<\cL X^{\gamma_2(1-\gamma_1)}}}|\mathfrak{T}_{II,2}|
     \ll X^{\gamma_1\gamma_2-2\varepsilon}.
	\end{align*}
Here we also remark that $12\gamma_1\gamma_2>11$ is derived from $15\gamma_1\gamma_2>14$.

   Therefore,
   $$
   \sum_{\substack{(k,h) \in \ZZ^2 \backslash (0,0) \\ |k|<\cL X^{1-\gamma_2} \\ |h|<\cL X^{\gamma_2(1-\gamma_1)}}}|\mathfrak{T}_{II}|
   \ll \sum_{\substack{(k,h) \in \ZZ^2 \backslash (0,0) \\ |k|<\cL X^{1-\gamma_2} \\ |h|<\cL X^{\gamma_2(1-\gamma_1)}}}|\mathfrak{T}_{II,1}|
   +\sum_{\substack{(k,h) \in \ZZ^2 \backslash (0,0) \\ |k|<\cL X^{1-\gamma_2} \\ |h|<\cL X^{\gamma_2(1-\gamma_1)}}}|\mathfrak{T}_{II,2}|
     \ll X^{\gamma_1\gamma_2-2\varepsilon}.
   $$
\end{proof}

We return to the proof. Combining the Vaughan identity (Lemma~\ref{lem:vaughan}), Lemma \ref{lemma: Gamma I}
and Lemma \ref{lemma: Gamma II} we find that
$$
\sum_{\substack{(k,h) \in \ZZ^2 \backslash (0,0) \\ |k|<\cL X^{1-\gamma_2} \\ |h|<\cL X^{\gamma_2(1-\gamma_1)}}}
|\Psi|\ll X^{\gamma_1\gamma_2-2\varepsilon}
$$
provided that $j\leqslant W^{1+2\varepsilon}$.
Recall that $W=X^{1-\gamma_1\gamma_2+\varepsilon}$. For the sum $\cF_{11}$ we have
\begin{align*}
	\sum_{\substack{(k,n) \in \ZZ^2 \backslash (0,0) \\ |k|<\cL X^{1-\gamma_2} \\ |n|<\cL X^{\gamma_2(1-\gamma_1)}}}|\cF_{11}|
	\ll\frac{1}{W}\sum_{j\leqslant \varepsilon W}\Bigg|
	\sum_{w=0}^{W-1}\mathbf{e}\left(\frac{-wj}{W}\right)\Bigg|
	\sum_{\substack{(k,n) \in \ZZ^2 \backslash (0,0) \\ |k|<\cL X^{1-\gamma_2} \\ |n|<\cL X^{\gamma_2(1-\gamma_1)}}}|\Psi|\ll X^{1-\varepsilon}.
\end{align*}
For the sum $\cF_{12}$ we also have
$$
\sum_{\substack{(k,n) \in \ZZ^2 \backslash (0,0) \\ |k|<\cL X^{1-\gamma_2} \\ |n|<\cL X^{\gamma_2(1-\gamma_1)}}}|\cF_{12}|\ll
\sum_{\varepsilon W<j\leqslant W^{1+2/r}}
\left(\frac{W}{j}\right)^{r+1}
X^{\gamma_1\gamma_2-2\varepsilon}\ll X^{1-\varepsilon}.
$$
For the sum $\cF_{13}$ we have
\begin{align*}
	\sum_{\substack{(k,n) \in \ZZ^2 \backslash (0,0) \\ |k|<\cL X^{1-\gamma_2} \\ |n|<\cL X^{\gamma_2(1-\gamma_1)}}}|\cF_{13}|
	\ll \sum_{j>W^{1+2/r}}\left(\frac{W}{j}\right)^{r+1}X^{2-\gamma_1\gamma_2}
	\ll W^{-1}X^{2-\gamma_1\gamma_2}\ll X^{1-\varepsilon}.
\end{align*}
Therefore,
\begin{align}\label{eq: bound for F1}
	\sum_{\substack{(k,n) \in \ZZ^2 \backslash (0,0) \\ |k|<\cL X^{1-\gamma_2} \\ |n|<\cL X^{\gamma_2(1-\gamma_1)}}}|\cF_{1}|
	\ll \sum_{\substack{(k,n) \in \ZZ^2 \backslash (0,0) \\ |k|<\cL X^{1-\gamma_2} \\ |n|<\cL X^{\gamma_2(1-\gamma_1)}}}\left(|\cF_{11}|+|\cF_{12}|+|\cF_{13}|\right)
	\ll X^{1-\varepsilon}.
\end{align}
This proves Proposition \ref{F1}.

\subsection{Proof of Proposition~\ref{F2} and \ref{F3}}

It is clear that
\begin{align}\label{eq: bound for F2}
	\sum_{\substack{(k,h) \in \ZZ^2 \backslash (0,0) \\ |k|<\cL X^{1-\gamma_2} \\ |h|<\cL X^{\gamma_2(1-\gamma_1)}}}|\cF_2|
	\ll \sum_{\substack{(k,h) \in \ZZ^2 \backslash (0,0) \\ |k|<\cL X^{1-\gamma_2} \\ |h|<\cL X^{\gamma_2(1-\gamma_1)}}}\frac{rX\log X}{W}\ll rX^{1-\varepsilon/2}.
\end{align}
Hence we finish the proof of proposition~\ref{F2}.

Now we proof Proposition \ref{F3}. We divide the sum $\cF_3$ into four parts
$$
\cF_3=\cF_{31}+\cF_{32}+\cF_{33}+O\left(rX^{\gamma_1\gamma_2-\varepsilon}\right),
$$
where $\cF_{31}$ is the sum of the terms $j\leqslant \varepsilon W$, $\cF_{32}$
is the sum of terms $\varepsilon W<j\leqslant W^{1+\varepsilon}$, $\cF_{33}$ is the
sum of terms $j>W^{1+\varepsilon}$, and the error term is taken from $j=0$. We start with
$\cF_{33}$. It follows from $r=\varepsilon^{-1}$ that
\begin{align*}
\cF_{33}\ll \frac{r}{W}\sum_{j>W^{1+\varepsilon}}\left(\frac{W}{j}\right)^{r+1}X
\ll r X^{\gamma_1\gamma_2-\varepsilon}.
\end{align*}

It remains to estimate sums of the shape
$$
\sum_{X < m \leqslant aX}\Lambda(m)\mathbf{e}\left(j(m+1)^{\gamma_2}\right)
$$
when $j\leqslant W^{1+\varepsilon}$.
We repeat the standard process (see reference \cite[P.46-52]{GrahamK1991}) to treat it.
If $a_m\ll X^{\varepsilon}$, $b_n=1$ or $b_n=\log n$, and $M<X^{1/3}$,
using the exponent pair $(1/2,1/2)$ we have
\begin{align*}
\mathop{\sum_{m\sim M}\sum_{n\sim N}
}_{X < mn \leqslant aX}a_mb_n\mathbf{e}\left(j(mn+1)^{\gamma_2}\right)
\ll \left(j^{1/2}X^{(3\gamma_2+2)/6}+j^{-1}X^{1-\gamma_2}\right)X^{2\varepsilon}.
\end{align*}
If $a_m\ll X^{\varepsilon}$, $b_n\ll X^{\varepsilon}$, and $X^{1/2}<X<X^{2/3}$,
using the exponent pair $(1/2,1/2)$ we also have
\begin{align*}
\mathop{\sum_{m\sim M}\sum_{n\sim N}
}_{X < mn \leqslant aX}a_mb_n\mathbf{e}\left(j(mn+1)^{\gamma_2}\right)
\ll \left(X^{5/6}+j^{1/6}X^{(2\gamma_2+9)/12}
+j^{-1/2}X^{(2-\gamma_2)/2}\right)X^{2\varepsilon}.
\end{align*}

From the Vaughan identity and the fact that
\begin{align}\label{eq: condition 1}
\left\{\begin{array}{ll}
9\gamma_1\gamma_2-3\gamma_2-5>0, & \\
6\gamma_1\gamma_2-5>0, & \\
14\gamma_1\gamma_2-2\gamma_2-11>0, &
\end{array}\right.
\end{align}
we have
$$
\sum_{X < m \leqslant aX}\Lambda(m)\mathbf{e}\left(j(m+1)^{\gamma_2}\right)
\ll X^{\gamma_1\gamma_2-\varepsilon}
$$
provided that $j\leqslant W^{1+\varepsilon}$. It follows that
\begin{align*}
\cF_{31}\ll \frac{r}{W}\sum_{j\leqslant \varepsilon W}
X^{\gamma_1\gamma_2-\varepsilon}\ll rX^{\gamma_1\gamma_2-\varepsilon}
\end{align*}
and
\begin{align*}
\cF_{32}\ll\frac{r}{W}\sum_{\varepsilon W<j\leqslant W^{1+\varepsilon}}
\left(\frac{W}{j}\right)^{r+1}X^{\gamma_1\gamma_2-\varepsilon}
\ll r^{r+1} X^{\gamma_1\gamma_2-\varepsilon}.
\end{align*}

Therefore, we conclude that
\begin{align}\label{eq: bound for F3}
\sum_{\substack{(k,h) \in \ZZ^2 \backslash (0,0) \\ |k|<\cL X^{1-\gamma_2} \\ |h|<\cL X^{\gamma_2(1-\gamma_1)}}}|\cF_3|
&\ll \sum_{\substack{(k,h) \in \ZZ^2 \backslash (0,0) \\ |k|<\cL X^{1-\gamma_2} \\ |h|<\cL X^{\gamma_2(1-\gamma_1)}}}\left(|\cF_{31}|+|\cF_{32}|+|\cF_{33}|
+rX^{\gamma_1\gamma_2-\varepsilon}\right)\nonumber\\
&\ll X^{1-\varepsilon/2}.
\end{align}
This completes the proof of Proposition \ref{F2} and Proposition \ref{F3}.

\section*{Acknowledgement}
We thank Prof. Yuchen Ding and Prof. Jinjiang Li for valuable conversations. 

This work was supported by the National Natural Science Foundation of China (No. 11901447, No. 12501011),
the Natural Science Foundation of Shaanxi Province (No. 2024JC-YBMS-029, No. 2025JC-YBQN-015),
the Shaanxi Fundamental Science Research Project for Mathematics and Physics (No. 22JSY006),
and the Scientific Research Program Funded by Shaanxi Provincial Education Department (No. 24JK0680).


\begin{thebibliography}{99}


\bibitem{BackerBBSW2013}
R. C. Baker, W. D. Banks, J. Br\"{u}dern, I. E. Shparlinski and A. J. Weingartner,
Piatetski-Shapiro sequences.
\emph{Acta Arith.}, 157 (2013), no. 1, 37--68.

\bibitem{Baker1986}
R. C. Baker,
Diophantine inequalities. London Mathematical Society Monographs. New Series, 1. Oxford Science Publications. The Clarendon Press, Oxford University Press, New York, 1986.
	
\bibitem{Baker2014}
R. C. Baker,
The intersection of Piatetski-Shapiro sequences.
\emph{Mathematika},
60 (2014), no. 2, 347--362.

\bibitem{BI1986}
E. Bombieri and H. Iwaniec,
On the order of $\zeta(\frac{1}{2} + it)$.
\emph{Ann. Scuola Norm. Sup. Pisa Cl. Sci. (4)} 13 (1986), no. 3, 449--472.

\bibitem{Bour}
J. Bourgain,
Decoupling, exponential sums and the Riemann zeta function.
\emph{J. Amer. Math. Soc.} 30 (2017), no. 1, 205--224.

\bibitem{CaoZ1998}
X. D. Cao and W. G. Zhai,
The distribution of square-free numbers of the form $[n^c]$.
\emph{J. Th\'{e}or. Nombres Bordeaux},
10 (1998), no. 2, 287--299.

\bibitem{Daven}
H.~Davenport,
Multiplicative number theory.
Graduate Texts in Mathematics, \textbf{74}. Springer-Verlag, New York-Berlin, 1980.

\bibitem{FI1989}
\'{E}. Fouvry and H. Iwaniec, 
Exponential sums with monomials.
\emph{J. Number Theory} 33 (1989), no. 3, 311--333.

\bibitem{GrahamK1991}
S. W. Graham and G. Kolesnik,
van der Corput's method of exponential sums,
London Mathematical Society Lecture Note Series, 126. Cambridge University Press, Cambridge, 1991.

\bibitem{Guo2021}
V. Z. Guo,
Almost primes in Piatetski-Shapiro sequences.
\emph{AIMS Math.}, 6 (2021), no. 9, 9536--9546.

\bibitem{Harman1993}
G. Harman,
Small fractional parts of additive forms.
\emph{Philos. Trans. R. Soc. Lond. A} 345 (1993), no. 1676, 327--338.

\bibitem{Heath-Brown}
D. R. Heath-Brown,
Prime numbers in short intervals and a generalized Vaughan identity.
\emph{Canad. J. Math.}, 34 (1982), no. 6, 1365--1377.

\bibitem{Heath-Brown1983}
D. R. Heath-Brown,
The Pjateckii-\v{S}apiro prime number theorem.
\emph{J. Number Theory}, 16 (1983), no. 2, 242--266.

\bibitem{IwaniecK}
H. Iwaniec and E. Kowalski,
Analytic Number Theory, American Mathematical Society Colloquium Publications, 53, American Mathematical Society, Providence, RI, 2004.

\bibitem{Kolesnik2003}
G. Kolesnik,
Estimation of some exponential sums containing the fractional part function and some
other ``non-standard'' exponential sums.
\emph{Acta Arith.}, 108 (2003), no. 4, 315--326.

\bibitem{Leitmann1982}
D. Leitmann,
Durchschnitte von Pjateckij-Shapiro-Folgen.
\emph{Monatsh. Math.}, 94 (1982), no. 1, 33--44.



\bibitem{Mont1994}
H. L. Montgomery,
Ten lectures on the interface between analytic number theory and harmonic analysis. CBMS Regional Conference Series in Mathematics, 84. Published for the Conference Board of the Mathematical Sciences, Washington, DC; by the American Mathematical Society, Providence, RI, 1994.

\bibitem{Shapiro1953}	
I. I. Piatetski-Shapiro,
On the distribution of prime numbers in sequences of the form $[f(n)]$.
\emph{Math. Sb.}, 33(75) (1953), 559--566.

\bibitem{RivatS2001}
J. Rivat and P. Sargos,
Nombres premiers de la forme $[n^c]$.
\emph{Canad. J. Math.}, 53 (2001), no. 2, 414--433.

\bibitem{RivatW2001}
J. Rivat and J. Wu,
Prime numbers of the form $[n^c]$.
\emph{Glasg. Math. J.}, 43 (2001), no. 2, 237--254.

\bibitem{RS2006}
O. Robert and P. Sargos,
Three-dimensional exponential sums with monomials.
\emph{J. Reine Angew. Math.}, 591 (2006), 1--20.

\bibitem{Sargos1995}
P. Sargos,
Points entiers au voisinage d'une courbe, sommes trigonom\'{e}triques courtes et
paires d'exposants.
\emph{Proc. London Math. Soc.},  70 (1995), no. 2, 285--312.



\bibitem{Sirota1983}
E. R. Sirota,
Laws for the distribution of primes of the form $p=[n^c]=[m^d]$ in arithmetic progressions.
\emph{Zap. Nauchn. Sem. Leningrad Otdel. Mat. Inst. Steklov}, 121 (1983), 94--102.


\bibitem{Zhai1999}
W. G. Zhang,
On the $k$-dimensional Piatetski-Shapiro prime number theorem.
\emph{Sci. China Ser. A}, 42 (1999), no. 11, 1173--1183.




\end{thebibliography}
\end{document}